\documentclass[12pt]{amsart}
\usepackage{a4wide}
\usepackage[T1]{fontenc}
\usepackage{amssymb,amsmath,amsthm,latexsym}
\usepackage{mathrsfs}
\usepackage[usenames,dvipsnames]{color}
\usepackage{euscript}
\usepackage{graphicx}
\usepackage{mdwlist}
\usepackage{enumerate}
\usepackage{mathtools,dsfont,wasysym}
\usepackage{stmaryrd}
\usepackage{hyperref}
\hypersetup{colorlinks=true,  linkcolor=blue, citecolor=red, urlcolor=cyan}

\newtheorem{theorem}{Theorem}[section]
\newtheorem{lemma}[theorem]{Lemma}
\newtheorem{corollary}[theorem]{Corollary}

\newtheorem{fact}[theorem]{Fact}
\newtheorem{claim}[theorem]{Claim}
\newtheorem{problem}[theorem]{Problem}
\newcounter{nstep}
\newtheorem{step}[nstep]{Step}
\newcounter{maintheorem}

\newtheorem{mainth}[maintheorem]{Theorem}
\theoremstyle{remark}
\newtheorem{remark}[theorem]{Remark}
\theoremstyle{definition}

\numberwithin{equation}{section}
\makeatother

\newcommand{\R}{\mathbb{R}}
\newcommand{\N}{\mathbb{N}}
\newcommand{\e}{\varepsilon}
\newcommand{\ro}{\varrho}
\newcommand{\p}{\varphi}
\newcommand{\n}{\left\Vert\cdot\right\Vert}
\newcommand{\nn}[1]{{\left\vert\kern-0.25ex\left\vert\kern-0.25ex\left\vert #1 \right\vert\kern-0.25ex \right\vert\kern-0.25ex \right\vert}}
\newcommand{\ceil}[1]{\left\lceil #1 \right\rceil}

\renewcommand{\leq}{\leqslant}
\renewcommand{\geq}{\geqslant}
\DeclareMathOperator{\supp}{supp}
\DeclareMathOperator{\diam}{diam}
\DeclareMathOperator{\dist}{dist}
\DeclareMathOperator{\spn}{span}

\newcounter{smallromans}

\newenvironment{romanenumerate}
{\begin{list}{{\normalfont\textrm{(\roman{smallromans})}}}%
  {\usecounter{smallromans}\setlength{\itemindent}{0cm}%
   \setlength{\leftmargin}{5.5ex}\setlength{\labelwidth}{5.5ex}%
   \setlength{\topsep}{.5ex}\setlength{\partopsep}{.5ex}%
   \setlength{\itemsep}{0.1ex}}}%
{\end{list}}

\newcommand{\X}{\mathcal{X}}
\newcommand{\Y}{\mathcal{Y}}
\newcommand{\U}{\mathcal{U}}
\newcommand{\V}{\mathcal{V}}
\newcommand{\uS}{\mathcal{S}}
\newcommand{\B}{\mathcal{B}}
\newcommand{\Bo}{\B_\Y^{\mathsf{o}}}
\newcommand{\D}{\mathcal{D}}
\renewcommand{\P}{\mathcal{P}}
\newcommand{\Fn}{{\mathcal{F}^n}}
\newcommand{\F}{{\mathcal{F}^{<\omega}}}
\renewcommand\qedsymbol{$\blacksquare$} 

\makeatletter
\renewcommand{\tocsection}[3]{%
	\indentlabel{\@ifnotempty{#2}{\bfseries\ignorespaces#1 #2\quad}}\bfseries#3}
\renewcommand{\tocsubsection}[3]{%
	\indentlabel{\@ifnotempty{#2}{\ignorespaces#1 #2\quad}}#3}

\newcommand\@dotsep{4.5}
\def\@tocline#1#2#3#4#5#6#7{\relax
	\ifnum #1>\c@tocdepth 
	\else
	\par \addpenalty\@secpenalty\addvspace{#2}%
	\begingroup \hyphenpenalty\@M
	\@ifempty{#4}{%
		\@tempdima\csname r@tocindent\number#1\endcsname\relax
	}{%
		\@tempdima#4\relax
	}%
	\parindent\z@ \leftskip#3\relax \advance\leftskip\@tempdima\relax
	\rightskip\@pnumwidth plus1em \parfillskip-\@pnumwidth
	#5\leavevmode\hskip-\@tempdima{#6}\nobreak
	\leaders\hbox{$\m@th\mkern \@dotsep mu\hbox{.}\mkern \@dotsep mu$}\hfill
	\nobreak
	\hbox to\@pnumwidth{\@tocpagenum{\ifnum#1=1\bfseries\fi#7}}\par
	\nobreak
	\endgroup
	\fi}
\AtBeginDocument{%
	\expandafter\renewcommand\csname r@tocindent0\endcsname{0pt}
}
\def\l@subsection{\@tocline{2}{0pt}{2.5pc}{5pc}{}}
\makeatother

\begin{document}
\title[Smooth and polyhedral norms via fundamental biorthogonal systems]{Smooth and polyhedral norms via\\ fundamental biorthogonal systems}

\author[S.~Dantas]{Sheldon Dantas}
\address[S.~Dantas]{Departament de Matem\`atiques and Institut Universitari de Matem\`atiques i Aplicacions de Castell\'o (IMAC), Universitat Jaume I, Campus del Riu Sec. s/n, 12071 Castell\'o, Spain \newline
\href{https://orcid.org/0000-0001-8117-3760}{ORCID: \texttt{0000-0001-8117-3760}}}
\email{\texttt{dantas@uji.es}}

\author[P.~H\'ajek]{Petr H\'ajek}
\address[P.~H\'ajek]{Department of Mathematics\\Faculty of Electrical Engineering\\Czech Technical University in Prague\\Technick\'a 2, 166 27 Prague 6\\ Czech Republic}
\email{hajek@math.cas.cz}

\author[T.~Russo]{Tommaso Russo}
\address[T.~Russo]{Institute of Mathematics\\ Czech Academy of Sciences\\ \v{Z}itn\'a 25, 115 67 Prague 1\\ Czech Republic\\ and Department of Mathematics\\Faculty of Electrical Engineering\\Czech Technical University in Prague\\Technick\'a 2, 166 27 Prague 6\\ Czech Republic \newline
\href{https://orcid.org/0000-0003-3940-2771}{ORCID: \texttt{0000-0003-3940-2771}}}
\email{russo@math.cas.cz, russotom@fel.cvut.cz}

\thanks{Research of S.~Dantas was supported by the Spanish AEI Project PID2019 - 106529GB - I00 / AEI / 10.13039/501100011033 and also by Spanish AEI Project PID2021-122126NB-C33 / MCIN / AEI / 10.13039/501100011033 (FEDER).\\
P.~H\'ajek was supported in part by OPVVV CAAS CZ.02.1.01/0.0/0.0/16$\_$019/0000778.\\
T.~Russo was supported by the GA\v{C}R project 20-22230L; RVO: 67985840 and by Gruppo Nazionale per l'Analisi Matematica, la Probabilit\`a e le loro Applicazioni (GNAMPA) of Istituto Nazionale di Alta Matematica (INdAM), Italy.}

\date{\today}

\begin{abstract} Let $\mathcal{X}$ be a Banach space with a fundamental biorthogonal system and let $\mathcal{Y}$ be the dense subspace spanned by the vectors of the system. We prove that $\mathcal{Y}$ admits a $C^\infty$-smooth norm that locally depends on finitely many coordinates (LFC, for short), as well as a polyhedral norm that locally depends on finitely many coordinates. As a consequence, we also prove that $\mathcal{Y}$ admits locally finite, $\sigma$-uniformly discrete $C^\infty$-smooth and LFC partitions of unity and a $C^1$-smooth LUR norm. This theorem substantially generalises several results present in the literature and gives a complete picture concerning smoothness in such dense subspaces. Our result covers, for instance, every WLD Banach space (hence, all reflexive ones), $L_1(\mu)$ for every measure $\mu$, $\ell_\infty(\Gamma)$ spaces for every set $\Gamma$, $C(K)$ spaces where $K$ is a Valdivia compactum or a compact Abelian group, duals of Asplund spaces, or preduals of Von Neumann algebras. Additionally, under Martin Maximum {\sf MM}, all Banach spaces of density $\omega_1$ are covered by our result.
\end{abstract}
\maketitle
\tableofcontents

\section{Introduction}
The topic of smooth approximation is one of the classical themes in analysis. In the setting of Banach spaces the problem has several aspects, among which: the existence of smooth partitions of unity, smooth extensions, approximation of continuous or Lipschitz functions by smooth ones, smooth renormings, approximation with smooth norms, study of polynomials and spaces of polynomials, and so on. For an introduction to these directions of research we refer to the monographs \cite{DGZ, F book, HJ, P book}.

It is by now a well-known fact that the existence of a smooth norm (or more generally a smooth bump) on a Banach space $\X$ has several deep structural consequences for the space. For example, the presence of a $C^1$-smooth bump implies that the space is Asplund, \cite{F}; the presence of an LFC bump yields that the space is a $c_0$-saturated Asplund space, \cite{FZ, PWZ}. If $\X$ admits a $C^2$-smooth bump, then either it contains a copy of $c_0$, or it is super-reflexive with type $2$, \cite{FWZ}. Finally, if $\X$ admits a $C^\infty$-smooth bump and it contains no copy of $c_0$, then it has exact cotype $2k$, for some $k\in\N$, and it contains $\ell_{2k}$, \cite{D very smooth}. Each of these results involves at some point the completeness of the space $\X$, most frequently via the appeal to some form of variational principles, such as the Ekeland variational principle \cite{Ekeland}, Stegall's variational principle \cite{Stegall}, the Borwein--Preiss smooth variational principle \cite{BP}, or the compact variational principle \cite{DF}. It is therefore unclear whether any, possibly weaker, form of the above results could be valid for general normed spaces. In this direction, it was pointed out in \cite[p.~96]{BL} that it is not known whether $\X$ is an Asplund space provided the set where its norm fails to be Fr\'echet differentiable is `small' in some sense (also see \cite[Problem 148]{GMZ}). For example, it is unknown if there is a norm on $\ell_1$ that is Fr\'echet differentiable outside a countable union of hyperplanes.\smallskip

Nevertheless, some scattered results concerning normed spaces are present in the literature. Vanderwerff \cite{V} proved that every normed space with a countable algebraic basis admits a $C^1$-smooth norm; this result was later improved to obtain a $C^\infty$-smooth norm \cite{H smooth LFC}, a polyhedral norm \cite{DFH}, and an analytic one \cite{DHR}. These results and the previous discussion motivated \cite[Problem 149]{GMZ}, \cite{HR densely}, and recent research of the present authors \cite{DHR}, where the following problem was posed.

\begin{problem} Let $\X$ be a Banach space and $k\in\N\cup \{\infty, \omega\}$. Is there a dense subspace $\Y$ of $\X$ that admits a $C^k$-smooth norm?
\end{problem}

Although the problem is seemingly very general and ambitious, note that \cite{DHR} answers it in the positive for $\X$ separable and $k=\omega$. Moreover, in \cite{DHR} it was solved in the positive for $\ell_\infty$ and $k=\omega$, $\ell_1(\mathfrak{c})$ and $k=\omega$, and spaces with long unconditional bases and $k=\infty$. It is worth pointing out that in these results the density of smooth norms cannot be guaranteed in general. The main contribution of the present paper is a vast generalisation of the previous results by means of the following theorem. 

\begin{mainth}\label{Th A} Let $\X$ be a Banach space with a fundamental biorthogonal system $\{e_\alpha;\p_\alpha\}_{\alpha\in\Gamma}$. Consider the dense subspace $\Y$ of $\X$ given by $\Y\coloneqq\spn \{e_\alpha\}_{\alpha\in\Gamma}$. Then:
\begin{romanenumerate}
    \item \label{Th A: poly} $\Y$ admits a polyhedral and LFC norm,
    \item \label{Th A: smooth} $\Y$ admits a $C^\infty$-smooth and LFC norm,
    \item \label{Th A: bump} $\Y$ admits a $C^\infty$-smooth and LFC bump,
    \item \label{Th A: partitions of unity} $\Y$ admits locally finite, $\sigma$-uniformly discrete $C^\infty$-smooth and LFC partitions of unity,
    \item \label{Th A: C1 LUR} $\Y$ admits a $C^1$-smooth LUR norm.
\end{romanenumerate}
Moreover, norms as in (\ref{Th A: poly}), (\ref{Th A: smooth}), and (\ref{Th A: C1 LUR}) are dense in the set of all equivalent norms on $\Y$.
\end{mainth}

The main novel parts of Theorem \ref{Th A} are claims (\ref{Th A: poly}) and (\ref{Th A: smooth}), whose proof constitutes the core of the paper and is presented in Section \ref{Sec: main proof}. Clauses (\ref{Th A: bump})--(\ref{Th A: C1 LUR}) follow from the former ones via known results or adaptations of known techniques. More precisely, (\ref{Th A: bump}) is an obvious consequence of (\ref{Th A: smooth}), while the existence of locally finite, $\sigma$-uniformly discrete $C^\infty$-smooth partitions of unity follows from (\ref{Th A: smooth}) via \cite[Corollary 6]{J} or \cite[Theorem 2]{Haydon}. The fact that functions in the partition of unity can be chosen to be LFC follows by inspection of the proof of \cite{Haydon} as we will briefly discuss in Section \ref{Sec: partitions}. Finally, (\ref{Th A: C1 LUR}) follows from (\ref{Th A: smooth}) by using ideas from \cite{HP} and we shall explain this in Section \ref{Sec: C1 LUR}.\smallskip

Let us point out that Theorem \ref{Th A} draws a complete picture concerning smoothness in the sense that it implies the existence of smooth norms, norm approximation by smooth norms, $C^1$-smooth LUR norms, and the existence of partitions of unity, which are instrumental for the smooth approximation of continuous or Lipschitz functions (see, \emph{e.g.}, \cite[\S VIII.3]{DGZ} or \cite[Chapter 7]{HJ}). The unique part of the result where one might ponder possible improvements is (\ref{Th A: C1 LUR}), where it is natural to ask whether $\Y$ admits $C^k$-smooth LUR norms for some $k\geq 2$. Nonetheless, this is not the case in general, since a normed space with a $C^2$-smooth LUR norm has super-reflexive completion (Theorem \ref{Thm: C2 LUR super-reflexive}). Hence, in general it is not possible to replace $C^1$-smoothness with higher-order smoothness in (\ref{Th A: C1 LUR}) (even in the separable case). Although Theorem \ref{Thm: C2 LUR super-reflexive} is a more or less formal consequence of \cite[Theorem 3.3(ii)]{FWZ}, it is of notable importance in our context since it is one of the few instances where the existence of a smooth norm on an incomplete normed space bears structural consequences for the space.

As a simple consequence of Theorem \ref{Th A}, we can obtain one further such instance. Indeed, if a Banach space $\X$ admits a fundamental biorthogonal system, then densely many norms on $\X$ are $C^\infty$-smooth and LFC on a dense and open subset of $\X$ (Corollary \ref{Cor: open dense}). On the one hand, this result should be compared to the classical characterisation of Asplund spaces, that a Banach space $\X$ is Asplund if and only if every norm on $\X$ is Fr\'echet differentiable on a dense $G_\delta$ set. On the other hand, in the particular case of Banach spaces with fundamental biorthogonal systems, it generalises Moreno's result that every Banach space admits a norm that is Fr\'echet differentiable on a dense open set, \cite{Moreno open dense} (with Moreno's argument it doesn't seem possible to obtain the density of such norms). \smallskip

We now discuss how general our results are and compare them to the literature. If we restrict our attention to separable Banach spaces, a classical result due to Marku\v{s}evi\v{c} \cite{Markushevich} asserts that every separable Banach space admits an M-basis; hence, Theorem \ref{Th A} applies to every separable Banach space. Therefore, our result generalises simultaneously \cite{H smooth LFC}, where a $C^\infty$-smooth LFC norm is constructed in every normed space with a countable algebraic basis, and \cite{DFH}, where a polyhedral norm is constructed in such spaces. Here we should observe that, for a normed space, admitting a countable algebraic basis is equivalent to being the linear span of the vectors of an M-basis, again by \cite{Markushevich}. On the other hand, in \cite{DHR} an analytic norm is also constructed in normed spaces with a countable algebraic basis, while in our result it is not possible in general to obtain analytic norms, \cite[Theorem 3.10]{DHR}.

For non-separable Banach spaces the problem has only been faced in \cite[Theorem B]{DHR}, where a $C^\infty$-smooth norm is constructed in the linear span of every long unconditional Schauder basis (and in \cite[Theorem A]{DHR}, concerning the concrete spaces $\ell_\infty$ and $\ell_1(\mathfrak{c})$, as mentioned above). Once more, Theorem \ref{Th A} is substantially stronger, since we additionally obtain an approximation result, the LFC condition, polyhedral norms, partitions of unity, and $C^1$-smooth LUR norms. Moreover, the assumption on the space is much more general, for the existence of an unconditional basis is a rather strong assumption, while the existence of a fundamental biorthogonal system is a much weaker one, as we now discuss.

A large class of Banach spaces that admit a fundamental biorthogonal system (even an M-basis) is the class of Plichko spaces, \cite{HMVZ, KKL, Kalenda survey}. Such a class of Banach spaces contains all WLD Banach spaces, hence all WCG spaces and in particular all reflexive ones; besides, every $L_1(\mu)$ space and every $C(K)$ space, where $K$ is a Valdivia compactum or an Abelian compact group, is a Plichko space (see, \emph{e.g.}, \cite[\S 6.2 and \S 5.1]{Kalenda survey}, \cite{Kalenda group}). More generally Kalenda \cite{Kalenda PLMS} recently proved that every Banach space with a projectional skeleton admits a (strong) M-basis. Among the Banach spaces that admit a projectional skeleton we could additionally mention duals of Asplund spaces \cite{Kubis}, preduals of Von Neumann algebras \cite{BHK}, or preduals of JBW$^*$-triples \cite{BHKPP}. Additionally, there are several examples of concrete Banach spaces where a fundamental biorthogonal system can be constructed, for example: $\ell_\infty(\Gamma)$ for every set $\Gamma$ \cite{DJ}, $\ell_\infty^c(\Gamma)$ when $|\Gamma|\leq\mathfrak{c}$ \cite{GK, plichko no fund}, or $C([0,\eta])$ for every ordinal $\eta$ (this is standard, see for example \cite[Proposition 5.11]{Kalenda PLMS}). More generally, $C(T)$ has an M-basis, for every tree $T$, \cite[\S 5.3]{Kalenda PLMS}. Moreover, it is proved in \cite{DJ} that a Banach space $\X$ with ${\rm dens}\, \X=\kappa$ admits a fundamental biorthogonal system provided that $\X$ has a WCG quotient of density $\kappa$. Similarly, Plichko \cite{plichko bdd fund} proved that a Banach space $\X$ with ${\rm dens}\, \X=\kappa$ admits a fundamental biorthogonal system if and only if $\X$ has a quotient of density $\kappa$ with a long Schauder basis. Finally, it is consistent with {\sf ZFC}, and in particular true under Martin Maximum {\sf MM}, that every Banach space of density $\omega_1$ admits a fundamental biorthogonal system, \cite{T MA}.

Let us add one more comment concerning the space $\ell_\infty$. On the one hand, in \cite[Theorem 3.1]{DHR} an analytic norm is constructed in the dense subspace of $\ell_\infty$ comprising all sequences that attain finitely many values; while Theorem \ref{Th A} only gives a $C^\infty$-smooth norm in a subspace with a less explicit description. On the other hand, Theorem \ref{Th A} also yields the LFC condition, a polyhedral norm, a $C^1$-smooth LUR norm, and partitions of unity; additionally, it holds for $\ell_\infty(\Gamma)$ for every set $\Gamma$.\smallskip

Several months after the present research was completed, the authors obtained the following result \cite{DHR pq}, related to Theorem \ref{Th A}. If $1\leq p<\infty$, the dense subspace $\Y_p\coloneqq \bigcup_{0<q<p} \ell_q(\Gamma)$ of $\ell_p(\Gamma)$ admits a $C^\infty$-smooth and LFC norm. The interest of the result is that $\Y_p$ is not the linear span of a biorthogonal system and it has linear dimension equal to that of $\ell_p(\Gamma)$. \smallskip

Finally recall that it is in general unknown if, in a non-separable Banach space with a $C^k$-smooth norm, $C^k$-smooth norms are dense in the set of all equivalent norms. Among the few results available in the literature let us mention \cite{AFST, BS, Smith, SmTr}, where the problem is solved for spaces with a small boundary such as $c_0(\Gamma)$. In particular, the $C^k$-smooth approximation of norms is open in $\ell_2(\omega_1)$, or $C([0,\omega_1])$, while our Theorem \ref{Th A} gives the $C^\infty$-smooth approximation in some dense subspace of the said spaces. \smallskip

Our paper is organised as follows: Section \ref{Sec: prelim} contains the definitions of some notions that we will need and some auxiliary (known or folklore) results. Section \ref{Sec: main proof} is devoted to the main part of the proof of Theorem \ref{Th A} and we prove clauses (\ref{Th A: poly}) and (\ref{Th A: smooth}). A brief discussion of (\ref{Th A: partitions of unity}) is given in Section \ref{Sec: partitions}. Finally in Section \ref{Sec: C1 LUR} we discuss the existence of $C^k$-smooth LUR norms: we give the proof of Theorem \ref{Th A}(\ref{Th A: C1 LUR}) borrowing our methods from \cite{HP} and we prove Theorem \ref{Thm: C2 LUR super-reflexive}.

\section{Preliminary material}\label{Sec: prelim}
Our notation is standard as in, \emph{e.g.}, \cite{AK, DGZ, FHHMZ}. Throughout the paper we always consider normed spaces over the reals. We use the calligraphic font $\X, \Y$ for infinite-dimensional normed spaces and we denote by $F,G,H, \dots$ their finite-dimensional subspaces. For an infinite-dimensional normed space $\X$, we denote by $\X^*$, $\uS_\X$, and $\B_\X$ the dual space, the unit sphere, and the closed unit ball respectively; accordingly, the unit sphere and ball of the finite-dimensional normed space $F$ are $\uS_F$ and $\B_F$, respectively. We use the calligraphic notation for the unit sphere since we keep the letter $S$ for a generic slice as we will make extensive use of slices in our arguments. We write $\langle \p, x \rangle$ to denote the action of a functional $\p \in \X^*$ at a point $x \in \X$. When talking about norm approximations, hence in particular density of norms, we always refer to uniform approximation on bounded sets. More precisely, the assertion that a norm $\n$ on $\X$ can be approximated by norms with property $P$ means that, for every $\e>0$, there is a norm $\nn\cdot$ on $\X$ with property $P$ and such that $(1-\e)\nn\cdot \leq \n \leq (1+\e)\nn\cdot$.

\subsection{Convexity and slices}
A finite-dimensional normed space $F$ is \emph{polyhedral} if its unit ball is a polyhedron, \emph{i.e.}, it is a finite intersection of closed half-spaces; an infinite-dimensional normed space is polyhedral if every its finite-dimensional subspace is so. A normed space $(\X,\n)$ is \emph{locally uniformly rotund} (LUR, for short) if, for every $x\in \uS_\X$ and every sequence $(x_n)_{n\in\N}\subseteq \uS_\X$ with $\|x_n + x\|\to 2$ one has $x_n\to x$. A \emph{norming functional} for $x\in \X$ is a functional $\psi\in \uS_{\X^*}$ such that $\langle\psi,x\rangle=\|x\|$. A \emph{slice} of $\B_\X$ is a set of the form $S(\psi,\delta)\coloneqq\{y\in \B_\X\colon \langle\psi, y\rangle>1-\delta\}$, for some $\psi\in \uS_{\X^*}$ and $\delta>0$. For us it will be convenient to consider slices $S(\psi,\delta)$ where the functional $\psi$ attains its norm. Therefore, we write $S(x,\psi,\delta)$ to indicate the slice $S(\psi,\delta)$ whenever $\psi$ attains its norm at $x$. A point $x\in \uS_\X$ is \emph{strongly exposed} if there is a norming functional $\psi$ for $x$ such that $\diam\left( S(x,\psi, \delta)\right)\to 0$, as $\delta\to 0^+$. It is a standard fact that if $\X$ is LUR then every point of $\uS_\X$ is strongly exposed (see, for instance, \cite[Problem 8.27]{FHHMZ}), namely we have:

\begin{fact}\label{Fact: small slice} Let $\X$ be an LUR normed space, $x\in \uS_\X$, and $\psi\in \uS_{\X^*}$ be a norming functional for $x$. Then $\diam\left( S(x,\psi, \delta)\right)\to 0$, as $\delta\to 0^+$.
\end{fact}

We also collect here for future reference the following standard observation concerning slices; the proof is so simple that we include it here for the sake of completeness.

\begin{fact}\label{Fact: continuity slice} Let $\X$ be a normed space and let $\psi\in\uS_{\X^*}$ attain its norm at $x\in\uS_\X$. Then $\delta \mapsto \diam(S(x,\psi, \delta))$ is a continuous function on $(0,\infty)$.
\end{fact}
\begin{proof} Fix arbitrarily $\e>0$ and take $y,z\in S(x,\psi,\delta+\e)$. Consider the points $y_\lambda,z_\lambda\in \B_\X$ defined by $y_\lambda\coloneqq\lambda x+ (1-\lambda)y$ and $z_\lambda\coloneqq\lambda x+ (1-\lambda)z$, where $\lambda\coloneqq\frac{\e}{\delta+\e}\in(0,1)$. Then $y_\lambda,z_\lambda\in S(x,\psi, \delta)$, since
\begin{equation*}
    \langle\psi, y_\lambda\rangle= \langle\psi, \lambda x+ (1-\lambda)y\rangle > \lambda + (1-\lambda) (1-\delta-\e)= 1-\delta
\end{equation*}
(and analogously for $z_\lambda$). Hence,
\begin{equation*}
    \|y-z\|=\frac{1}{1-\lambda} \|y_\lambda- z_\lambda\| \leq\frac{1}{1-\lambda}\diam(S(x,\psi, \delta))= \left(1+\frac{\e}{\delta}\right) \diam(S(x,\psi, \delta)).
\end{equation*}
This yields $0\leq \diam(S(x,\psi,\delta+\e))- \diam(S(x,\psi,\delta))\leq 2\e/\delta$ and we are done.
\end{proof}

\subsection{Fundamental biorthogonal systems}
A \emph{biorthogonal system} in a normed space $\X$ is a system $\{e_\alpha; \p_\alpha\}_{\alpha\in\Gamma}$, with $e_\alpha\in\X$ and $\p_\alpha\in\X^*$, such that $\langle\p_\alpha, e_\beta \rangle=\delta_{\alpha,\beta}$ ($\alpha,\beta\in \Gamma$). A biorthogonal system is \emph{fundamental} if ${\rm span}\{e_\alpha\}_{\alpha\in\Gamma}$ is dense in $\X$; it is \emph{total} when ${\rm span}\{\p_\alpha\} _{\alpha\in\Gamma}$ is $w^*$-dense in $\X^*$. A \emph{Marku\v{s}evi\v{c} basis} (M-basis, for short) is a fundamental and total biorthogonal system. A biorthogonal system $\{e_\alpha; \p_\alpha\}_{\alpha\in\Gamma}$ is \emph{bounded} if there is $M<\infty$ with $\|x_\alpha\|\cdot \|\p_\alpha\|\leq M$ ($\alpha\in \Gamma$).

The following standard lemma, concerning distances of vectors from finite-dimensional subspaces in presence of a bounded biorthogonal system, will be used frequently in our argument. We refer to \cite[\S 1.2]{GL} for a more general treatment of such types of results.

\begin{lemma}\label{Lemma: bounded b.s.} Let $(\X,\n)$ be a normed space and $\{e_\alpha;\p_\alpha\}_{\alpha\in \Gamma}$ be a bounded biorthogonal system in $\X$ with $\|e_\alpha\|=1$, $\|\p_\alpha \|\leq M$ ($\alpha\in\Gamma$). Let $\alpha_1,\dots,\alpha_n, \beta_1,\dots,\beta_m\in \Gamma$ be such that $\alpha_i\neq \alpha_j$ and $\beta_i \neq \beta_j$ for $i\neq j$; set $F\coloneqq\spn \{e_{\alpha_1},\dots, e_{\alpha_n}\}$ and $G\coloneqq\spn \{e_{\beta_1},\dots, e_{\beta_m}\}$.
\begin{romanenumerate}
    \item\label{Item Lemma: bdd disjoint} If $\{\alpha_1,\dots,\alpha_n\} \cap \{\beta_1,\dots,\beta_m\}=\emptyset$ and $x\in \uS_F$, then
    \begin{equation*}
        \dist(x,G)\geq\frac{1}{nM}.
    \end{equation*}
    \item\label{Item Lemma: bdd general} More generally, for every $x\in F$,
    \begin{equation*}
        \dist(x,F\cap G)\leq nM \cdot \dist(x,G).    
    \end{equation*}
    \end{romanenumerate}
\end{lemma}
\begin{proof} (\ref{Item Lemma: bdd disjoint}) is a particular case of (\ref{Item Lemma: bdd general}) since the assumption of (\ref{Item Lemma: bdd disjoint}) gives $F\cap G=\{0\}$, so $\dist(x,F\cap G)=\|x\|$. For the proof of (\ref{Item Lemma: bdd general}), assume that, for some $k\geq 0$, $\alpha_1=\beta_1,\dots, \alpha_k=\beta_k$ and $\{\alpha_{k+1},\dots,\alpha_n\} \cap \{\beta_{k+1},\dots,\beta_m\} =\emptyset$. Fix $x\in F$ and let
\begin{equation*}
    \tilde{x}\coloneqq x- \sum_{j=k+1}^n \langle \p_{\alpha_j},x \rangle e_{\alpha_j}= \sum_{j=1}^k \langle \p_{\alpha_j},x \rangle e_{\alpha_j}\in F\cap G.
\end{equation*}
Moreover, taking $y\in G$ with $\dist(x,G)=\|x-y\|$, we can estimate
\begin{equation*}\begin{split}
    \dist(x,F\cap G)& \leq \|x-\tilde{x}\|= \left\|\sum_{j=k+1}^n \langle \p_{\alpha_j},x \rangle e_{\alpha_j}\right\| \leq \sum_{j=k+1}^n |\langle \p_{\alpha_j},x \rangle| \\
    &= \sum_{j=k+1}^n |\langle \p_{\alpha_j},x-y \rangle|\leq nM\cdot \|x-y\|=nM\cdot \dist(x,G).
\end{split}\end{equation*}
\end{proof}

In the proof of Theorem \ref{Th A} we shall need two important results concerning fundamental biorthogonal systems in Banach spaces, which we collect below. 

\begin{theorem}\label{Th: fundamental b.s.} Let $\X$ be a Banach space with a fundamental biorthogonal system $\{e_\alpha;\p_\alpha\}_{\alpha\in\Gamma}$ and let $\Y\coloneqq\spn \{e_\alpha\}_{\alpha\in\Gamma}$. Then
\begin{romanenumerate}
    \item\label{Item Th: bdd system} \cite{HM} there exists a bounded fundamental biorthogonal system $\{e'_\alpha;\p'_\alpha\}_{\alpha\in\Gamma}$ such that $\Y=\spn \{e'_\alpha\}_{\alpha\in\Gamma}$.
    \item\label{Item Th: LUR norm} \cite{T renorming, Moreno fund LUR} $\Y$ admits an LUR norm (which approximates the original norm of $\Y$).
\end{romanenumerate}
\end{theorem}

The claim in (\ref{Item Th: bdd system}) is stated and proved in \cite{HM} for M-bases only; however an inspection of the argument shows that it holds true for every fundamental biorthogonal system. Perhaps the simplest way to see this is to note that the argument in \cite{HM} never uses the completeness of the space; hence, if $\{e_\alpha;\p_\alpha \}_{\alpha\in\Gamma}$ is a fundamental biorthogonal system in a Banach space $\X$, it is an M-basis for the normed space $\Y\coloneqq\spn \{e_\alpha\}_{\alpha\in\Gamma}$ and we can apply \cite{HM} to the space $\Y$. Earlier partial results due to Plichko can be found in \cite{plichko bdd fund, plichko bbd strong Mbasis}. Part (\ref{Item Th: LUR norm}) is essentially Troyanski's renorming technique \cite{T renorming} (also see \cite[Chapter VII]{DGZ}, or \cite[Theorem 3.48]{HMVZ}); the same type of results can also be found in the literature under the name \emph{Deville's master lemma}, \emph{e.g.}, in \cite{OT1, OT2}. The formulation given here is stated explicitly, \emph{e.g.}, in \cite[Lemma 2.1]{Moreno fund LUR} and some variants of it are also used in \cite{Z, JM, GJM}. The assertion in parentheses concerning the density of LUR norms is standard, \cite[p.~52]{DGZ}. Finally note that Theorem \ref{Th A}(\ref{Th A: C1 LUR}) constitutes an improvement of Theorem \ref{Th: fundamental b.s.}(\ref{Item Th: LUR norm}).

\subsection{Smooth norms via Minkowski functionals} A norm $\n$ on $\X$ is \emph{$C^k$-smooth} if its $k$-th Fr\'echet derivative exists and it is continuous at every point of $\X\setminus\{0\}$ (equivalently, of $\uS_\X$). The norm $\n$ \emph{locally depends on finitely many coordinates} (is LFC, for short) on $\X$ if for each $x\in \uS_\X$ there exist an open neighbourhood $\U$ of $x$ and functionals $\p_1,\dots,\p_k \in\X^*$ such that $\|y\|=\|z\|$ for every $y,z\in\U$ with $\langle\p_j,y \rangle= \langle\p_j,z \rangle$ for every $j=1,\dots,k$.

A \emph{convex body} is a convex set with nonempty interior. A canonical way to build an equivalent norm on a normed space $(\X,\n)$ consists in building a bounded, symmetric convex body $\D$. Then $\D$ induces an equivalent norm on $\X$ via its Minkowski functional $\mu_\D$, defined by $\mu_\D(x)\coloneqq \inf\{t>0\colon x\in t\D\}$. If $\D$ is additionally closed, then the unit ball of $(\X,\mu_\D)$ coincides with $\D$ itself. Moreover, if $(1-\delta)\B_\X \subseteq \D \subseteq \B_\X$ for some $\delta > 0$, then $\n \leq \mu_\D \leq (1-\delta)^{-1} \n$. This approach is ubiquitous in smooth renorming, combined with the following standard lemma, a version of the Implicit Function theorem; the formulation given here follows from \cite[Lemma 5.23]{HJ}. We refer, \emph{e.g.}, to \cite{B boundary, DFH, FHZ lattice, FWZ, HR improving, PWZ} for some instances of uses of this technique.

\begin{lemma}\label{Lemma: implicit mu} Let $(\X,\n)$ be a normed space, $\D\neq\emptyset$ be an open, convex, and symmetric subset of $\X$ and $f\colon \D\to \R$ be even, convex, and continuous. Assume that there is $a>f(0)$ such that $\B\coloneqq\{f\leq a\}$ is bounded and closed in $\X$. If $f$ is $C^k$-smooth for some $k\in\N \cup \{\infty, \omega\}$ (resp.~LFC) on $\D$, then the Minkowski functional $\mu_\B$ of $\B$ is a $C^k$-smooth (resp.~LFC) norm on $\X$.
\end{lemma}

We will also need the following folklore lemma, a variation of the above result.

\begin{lemma}\label{Lemma: LFC} Let $\D$ be a bounded, symmetric, convex body in a normed space $\X$. Assume that for every $x\in \partial \D$ there are a neighbourhood $\U$ of $x$ and functionals $\p_1,\dots,\p_n\in\X^*$ such that
\begin{equation}\label{eq: Lemma LFC}
    \forall y\in\U\colon \quad y\in \D \,\Longleftrightarrow\, \langle \p_i,x \rangle\leq 1,\ \forall i=1,\dots, n.
\end{equation}
Then $\mu_\D$ is an LFC norm.
\end{lemma}
\begin{proof} Pick any $x\in \X$ with $\mu_\D(x)= 1$, namely $x\in \partial \D$. Let a neighbourhood $\U$ of $x$ and functionals $\p_1,\dots,\p_n\in\X^*$ be as in the statement of the lemma. There are an open neighbourhood $\V$ of $x$ and $\e>0$ such that $s^{-1}\cdot y\in\U$ for every $y\in\V$ and every $s\in \R$ with $1-\e\leq s\leq 1+\e$. Moreover, the continuity of $\mu_\D$ yields the existence of a neighbourhood $\mathcal{W}$ of $x$, $\mathcal{W}\subseteq \V$, such that $1-\e\leq \mu_\D(y)\leq 1+\e$ for every $y\in \mathcal{W}$. 

We claim that the neighbourhood $\mathcal{W}$ of $x$ and the functionals $\p_1,\dots,\p_n$ witness that $\mu_\D$ satisfies the LFC condition at $x$. Indeed, towards a contradiction, assume that there are $y,z\in \mathcal{W}$ such that $\langle \p_i,y \rangle=\langle \p_i,z\rangle$ for $i=1,\dots,n$, but $\mu_\D(y)<\mu_\D(z)$. Pick $s\in\R$ with $\mu_\D(y)<s<\mu_\D(z)$; then $1-\e\leq s\leq 1+\e$, by definition of $\mathcal{W}$. The definition of $\V$ now yields that $s^{-1}\cdot y, s^{-1}\cdot z\in \U$. Moreover, the definition of $\mu_\D$ gives $s^{-1}\cdot y\in \D$ and $s^{-1}\cdot z\notin \D$. Hence, by (\ref{eq: Lemma LFC}), we derive that there is $i_0\in\{1,\dots,n\}$ such that $\langle\p_{i_0}, s^{-1}\cdot z\rangle>1$, while $\langle\p_i, s^{-1}\cdot y\rangle\leq1$ for each $i=1,\dots,n$. However, this contradicts the fact that $\langle \p_{i_0}, y\rangle=\langle \p_{i_0}, z\rangle$ and concludes the proof.
\end{proof}

\section{Proof of the main result}\label{Sec: main proof}
The goal of this section is the proof of the core parts of our main result, items (\ref{Th A: poly}) and (\ref{Th A: smooth}) of Theorem \ref{Th A}. Before diving into the details of the proof, let us present here some of the main ideas involved in the argument.

Let $\Y\coloneqq\spn\{e_\alpha\}_{\alpha\in\Gamma}$ be the linear span of the fundamental biorthogonal system and $F\coloneqq\spn\{e_{\alpha_1},\dots, e_{\alpha_n}\}$ be a finite-dimensional subspace. By Theorem \ref{Th: fundamental b.s.}(\ref{Item Th: LUR norm}), we can assume that the norm on $\Y$ is LUR, so we can cover the unit sphere of $F$ with finitely many open slices of $\B_\Y$ with arbitrary small diameters. By removing such slices from $\B_\Y$ we end up with a convex body $\P_F$ whose intersection with $F$ is a polyhedron. After having performed such a construction in the single subspace $F$, one would like to `glue together' all such convex sets $\P_F$ in order to get the desired polyhedral norm on $\Y$. When trying to implement this idea, we face two main difficulties. The first one is that, if $G\coloneqq\spn\{e_{\beta_1},\dots, e_{\beta_k}\}$ is a subspace of $F$, the polyhedron $\P_F\cap F$ corresponding to $F$ intersected with $G$ should coincide with $\P_G\cap G$. In order to solve this, when choosing the slices in $F$, we only add those that have small intersection with the slices coming from proper subspaces of $F$. Moreover, the diameter of these `new' slices is much smaller than the one of the slices constructed before. This inductive construction will be carried out in Step \ref{Step: slices}. The second, and main, difficulty is that the slices corresponding to some other subspace $G\coloneqq\spn\{e_{\beta_1},\dots, e_{\beta_k}\}$ (now, not necessarily contained in $F$) could intersect $F$. Since this could happen for infinitely many of such subspaces $G$, there would be no way to assure that the ball of $F$ remains a polyhedron. The rather delicate choice of the slices in Step \ref{Step: slices} is justified by the need to circumvent such a problem. Indeed, in Step \ref{Step: compatible} we show that there even exists a neighbourhood of $\B_F$, that we will denote by $T(F,\theta_F/2)$, that `protects' $F$, in the sense that the other slices do not intersect such a neighbourhood. The existence of this neighbourhood will be the crucial ingredient in the proof of the LFC and for smoothing the norm in Steps \ref{Step: polyhedral norm} and \ref{Step: smooth norm} respectively.

\begin{proof}[Proof of Theorem \ref{Th A}(\ref{Th A: poly}) and (\ref{Th A: smooth})] Let $\X$ be a Banach space with a fundamental biorthogonal system $\{e_\alpha;\p_\alpha\}_{\alpha\in \Gamma}$, set $\Y\coloneqq\spn\{e_\alpha\}_{\alpha\in\Gamma}$, and choose any equivalent norm $\n$ on $\Y$. Theorem \ref{Th: fundamental b.s.}(\ref{Item Th: bdd system}) yields that $\Y$ is the linear span of a bounded (fundamental) biorthogonal system, hence it allows us to assume that $\{e_\alpha;\p_\alpha\}_{\alpha\in \Gamma}$ is bounded. Our task consists in building two norms, the former polyhedral and LFC, the latter $C^\infty$-smooth and LFC, that approximate $\n$. By Theorem \ref{Th: fundamental b.s.}(\ref{Item Th: LUR norm}) there exists an LUR norm on $\Y$ that approximates $\n$; hence, we can (and do) assume that $\n$ is already LUR on $\Y$. Finally, up to rescaling we can also assume that $\|e_\alpha\|=1$, $\|\p_\alpha\|\leq M$ ($\alpha\in\Gamma$), for some $M\geq 1$.\smallskip

We start by fixing one piece of notation. Denote, for $n\in\N$,
\begin{equation*}
    \Fn\coloneqq \big\{ \spn\{e_{\alpha_1},\dots, e_{\alpha_n}\}\colon \alpha_1,\dots,\alpha_n\in\Gamma,\, \alpha_i\neq \alpha_j \mbox{ for } i\neq j \big\}
\end{equation*}
\begin{equation*}
    \F\coloneqq\bigcup_{n\in\N} \Fn.
\end{equation*}
Note that $\F$ is a directed set by inclusion. Moreover, for $F\in\F$ and $\theta>0$, we consider the open `tubular' neighbourhood of $F$ in $\B_\Y$ defined by
\begin{equation*}
    T(F,\theta)\coloneqq\{y\in \B_\Y\colon \dist(y,F)<\theta\}.
\end{equation*}

\begin{step}\label{Step: slices} Construction of the slices.
\end{step}

We start by building a collection of slices parametrised by $\F$ that satisfy the conditions in the following claim.
\begin{claim}\label{Claim: slices} Let $\e>0$ be fixed. Then there are nets $(\e_F)_{F\in\F}$ and $(\theta_F)_{F\in\F}$ of positive reals and sets $(\Omega_F)_{F\in\F}$, where $\Omega_F$ is a finite set of slices of $\B_\Y$, such that:
\begin{romanenumerate}
    \item\label{Claim: i)} $\theta_F\leq \e_F\leq \e$ for every $F\in\F$,
    \item\label{Claim: ii)} $\e_F\leq \frac{1}{4nM}$ for every $F\in\Fn$,
    \item\label{Claim: iii)} $\e_F\leq \frac{1}{4nM} \theta_G$ for every $F\in\Fn$ and every $G\in\F$ with $G\subsetneq F$,
    \item\label{Claim: iv)} if $S\in\Omega_F$, then $-S \in\Omega_F$ as well,
    \item\label{Claim: v)} $\diam(S)<\e_F$ for every $S \in\Omega_F$,
    \item\label{Claim: vi)} if $S\in\Omega_F$, then $S$ is of the form $S= S(x,\psi,\delta)$, for some $\delta >0$, some functional $\psi\in \uS_{\Y^*}$ that is norming for $x$, where
	\begin{equation*}
	x \in \uS_F \setminus \bigcup_{\substack{G\in\F\\ G \subsetneq F}} T(G,\theta_G),
	\end{equation*}
    \item\label{Claim: vii)} Setting, for $F\in\F$, 
	\begin{equation*}
	\U_F \coloneqq \bigcup_{\substack{G\in\F \\ G \subseteq F}} \bigcup_{S \in \Omega_G} S,
	\end{equation*}
    we have $2\theta_F\leq  \dist (\uS_F, \B_\Y \setminus \U_F)$. In particular, $\uS_\Y \cap T(F, \theta_F) \subseteq \U_F$.
\end{romanenumerate}
\end{claim}

\begin{proof}[Proof of Claim \ref{Claim: slices}] \renewcommand\qedsymbol{$\square$} Without loss of generality, we assume that $\e<1$. We start by proving that the second clause in (\ref{Claim: vii)}) indeed follows from the first part. Pick any $x\in\uS_\Y\cap T(F,\theta_F)$ and find $w\in F$ with $\|x-w\|<\theta_F$. Since $\|x\|=1$, we have $\left|\|w\|-1\right|<\theta_F$; hence, the vector $\tilde{w}\coloneqq \frac{w}{\|w\|}$ satisfies $\|w-\tilde{w}\|<\theta_F$. Thus $\|\tilde{w}-x\|< 2\theta_F \leq \dist(\uS_F,\B_\Y \setminus \U_F)$. This inequality and $\tilde{w}\in\uS_F$ imply $x\in\U_F$, as desired.\smallskip

We now build $(\e_F)_{F\in\F}$, $(\theta_F)_{F\in\F}$, and $(\Omega_F)_{F\in\F}$ with the above properties and we argue by induction on $n\coloneqq\dim F$ (where $F \in \F$). To check the statement for $n=1$, pick any $F\in \mathcal{F}^1$, namely $F=\spn\{e_\alpha\}$, for some $\alpha\in\Gamma$. Set $\e_F\coloneqq\min\{\e,\frac{1}{4M}\}$ and let $\psi_\alpha\in\uS_{\Y^*}$ be a norming functional for $e_\alpha$. By Fact \ref{Fact: small slice} there is $\delta_\alpha$ such that the slice $S(e_\alpha, \psi_\alpha, \delta_\alpha)$ has diameter smaller than $\e_F$. Let $\Omega_F\coloneqq\{\pm S(e_\alpha, \psi_\alpha, \delta_\alpha)\}$ and $\U_F\coloneqq S(e_\alpha, \psi_\alpha, \delta_\alpha) \cup -S(e_\alpha, \psi_\alpha, \delta_\alpha)$. Then $\uS_F\subseteq \U_F$, so $\dist(\uS_F, \B_\Y\setminus \U_F)>0$ and we can finally choose $\theta_F\leq \min\{\e_F,\frac{1}{2}\dist(\uS_F, \B_\Y\setminus \U_F)\}$. With this construction, all conditions (\ref{Claim: i)})--(\ref{Claim: vii)}) are clearly satisfied. \smallskip

Now fix $n\geq2$ and assume inductively that $\e_F$, $\theta_F$, and $\Omega_F$ have already been defined for every $F\in\F$ with $\dim F\leq n-1$ and satisfy (\ref{Claim: i)})--(\ref{Claim: vii)}). Fix $F\in\Fn$ arbitrarily. There are only finitely many $G\in\F$ with $G\subsetneq F$ and $\theta_G$ has already been defined for each such $G$; hence, we can choose $\e_F>0$ such that 
\begin{equation*}
    \e_F\leq \frac{1}{4nM} \min\{\theta_G\colon G\in\F, G\subsetneq F\}.
\end{equation*}
This gives conditions (\ref{Claim: iii)}) and (\ref{Claim: ii)}), since $\theta_G\leq\e_G\leq \e<1$. Consider now the set 
\begin{equation*}
    \V\coloneqq \bigcup_{\substack{G\in\F\\ G \subsetneq F}} \bigcup_{S\in\Omega_G} S,
\end{equation*}
pick $x\in\uS_F\setminus \V$ and a norming functional $\psi_x$ for $x$. By Fact \ref{Fact: small slice} there exists $\delta_x>0$ such that $\diam (S(x,\psi_x, \delta_x))<\e_F$. The collection $\left\{S(x,\psi_x, \delta_x)\colon x\in\uS_F\setminus\V \right\}$ is an open cover of the compact set $\uS_F\setminus\V$, so we can extract a finite subcover $\left\{S(x_j,\psi_{x_j}, \delta_{x_j})\right\}_{j=1}^k$ of $\left\{S(x,\psi_x, \delta_x)\colon x\in\uS_F\setminus\V \right\}$. We set $\Omega_F\coloneqq \left\{\pm S(x_j,\psi_{x_j}, \delta_{x_j}) \right\}_{j=1}^k$. Then conditions (\ref{Claim: iv)}) and (\ref{Claim: v)}) are satisfied. (\ref{Claim: vi)}) is satisfied as well because 
\begin{equation*}
    \uS_\Y \cap \bigcup_{\substack{G\in\F\\ G \subsetneq F}} T(G,\theta_G) \subseteq \bigcup_{\substack{G\in\F\\ G \subsetneq F}} \bigcup_{S\in\Omega_G} S =\V
\end{equation*}
by (\ref{Claim: vii)}) of the inductive assumption.

Finally, we define $\U_F$ as in (\ref{Claim: vii)}); by construction $\U_F$ is an open subset of $\B_\Y$ that contains $\uS_F$. Hence, the closed sets $\uS_F$ and $\B_\Y\setminus \U_F$ are disjoint and, thus, they have positive distance ($\uS_F$ being compact). So, we can choose $\theta_F>0$ such that
\begin{equation*}
    \theta_F\leq \min\left\{\e_F,\frac{1}{2}\dist(\uS_F, \B_\Y\setminus \U_F)\right\}.
\end{equation*}
This yields conditions (\ref{Claim: i)}) and (\ref{Claim: vii)}) and concludes the inductive step.
\end{proof}

\begin{step}\label{Step: compatible} The convex bodies $\P_F$ and the compatibility condition.
\end{step}

Next, we use the family of slices from the previous step to build, for every $F\in\F$, a convex body $\P_F$. We prove the crucial fact (see Fact \ref{Fact: compatibility} below) that this construction is compatible, in the sense that the construction of $\P_F$ does not interfere with the one of $\P_G$, for any $F\neq G\in\F$.

For every $F\in\F$ and every $n\in\N$ we define
\begin{equation}\label{eq: def P_F}
    \P_F\coloneqq \B_\Y\setminus \bigcup_{\substack{G\in\F \\ G \subseteq F}} \bigcup_{S\in\Omega_G} S
\end{equation}
\begin{equation}\label{eq: def P_n}
    \P_n\coloneqq \bigcap_{F\in\Fn} \P_F = \bigcap_{\substack{G\in\F \\ \dim G \leq n}} \bigcap_{S\in\Omega_G} \B_\Y\setminus S.
\end{equation}

Let us note here the following monotonicity properties: $\P_F\subseteq \P_G$ whenever $F,G\in\F$, $G\subseteq F$. Hence, $\P_{n+1}\subseteq \P_n$ for $n\geq1$. We shall discuss further properties of such sets at the beginning of Step \ref{Step: polyhedral norm}, while we now turn to the main property of the construction.

\begin{fact}[Compatibility]\label{Fact: compatibility} For every $F\in \F$ and every $n\in\N$ with $\dim F\leq n$ we have
\begin{equation}\label{eq: compatibility}\tag{$\dagger$}
    \P_n \cap T(F,\theta_F/2) = \P_F \cap T(F,\theta_F/2).
\end{equation}
\end{fact}

\begin{proof}[Proof of Fact \ref{Fact: compatibility}] \renewcommand\qedsymbol{$\square$}
The `$\subseteq$' inclusion is clear since $\P_n\subseteq \P_F$ if $\dim F\leq n$, by the monotonicity properties mentioned above. So, we only need to prove that, for every $F\in\F$ and every $n\in\N$, one has
\begin{equation*}
    \P_F \cap T(F,\theta_F/2)\subseteq \P_n.
\end{equation*}
(Notice that checking the inclusion $\P_F \cap T(F,\theta_F/2)\subseteq \P_n$ for every $n\geq \dim F$ is equivalent to checking it for every $n\in\N$, since $(\P_n)_{n\in\N}$ is a decreasing sequence.) By the definition (\ref{eq: def P_n}) of $\P_n$ it is thus sufficient to prove that, for every $F,G\in\F$ and every $S\in\Omega_G$,
\begin{equation}\label{eq: star for compatibility} \tag{$*$}
    \P_F \cap T(F,\theta_F/2)\subseteq \B_\Y \setminus S.
\end{equation}
We prove this by induction on $n\coloneqq\max\{\dim F,\dim G\}$. Throughout the argument, we assume $F\neq G$ as (\ref{eq: star for compatibility}) is trivially true when $F=G$.\smallskip 

For the case $n=1$, assume by contradiction that there are $F\neq G\in\mathcal{F}^1$ and $S\in\Omega_G$ with $\P_F\cap T(F,\theta_F/2) \cap S\neq \emptyset$. Pick $x\in \P_F\cap T(F,\theta_F/2) \cap S$. Then $\dist(x,F)< \theta_F/2$. Moreover, by condition (\ref{Claim: vi)}) of Claim \ref{Claim: slices}, $S\in\Omega_G$ implies the existence of $z\in \uS_G \cap S$; hence, by (\ref{Claim: v)}), $\|x-z\|<\e_G$. Thus we get $\dist(z,F)<\e_G + \theta_F/2$. Therefore, Lemma \ref{Lemma: bounded b.s.}(\ref{Item Lemma: bdd disjoint}) and conditions (\ref{Claim: i)}) and (\ref{Claim: ii)}) give the following contradiction
\begin{equation*}
    \frac{1}{M}\leq \dist(z,F)<\e_G + \theta_F/2\leq \frac{1}{4M}+ \frac{1}{8M}\leq\frac{1}{2M}.
\end{equation*}\smallskip

Now fix $n\geq 2$ and assume by induction that (\ref{eq: star for compatibility}) holds for every $F,G\in\F$ with $\max\{\dim F,\dim G\}\leq n-1$ and every $S\in\Omega_G$. Take $F,G\in\F$ with $\max\{\dim F,\dim G\}= n$ and $S\in\Omega_G$. First of all, since $S\in\Omega_G$, (\ref{Claim: v)}) and (\ref{Claim: vi)}) respectively imply that $\diam(S)<\e_G$ and $S=S(z,\psi,\delta)$ for some
\begin{equation}\label{eq: z not in previous tubular nhoods}
    z \in \uS_G \setminus \bigcup_{\substack{H\in\F\\ H \subsetneq G}} T(H,\theta_H).
\end{equation}
Let us stress that, in particular, $z\in S\cap G$, fact that we shall use several times below. We distinguish two cases: $\dim F=n$, or $\dim G=n$ and $\dim F\leq n-1$.

\medskip
\emph{Case 1.} $\dim F=n$.
\smallskip

Set $k\coloneqq\dim G\leq n$ and, towards a contradiction, assume that there is $x\in \P_F \cap T(F,\theta_F/2)$ such that $x\in S$. Since $x,z\in S$, by (\ref{Claim: v)}), $\|x-z\|<\e_G$. Moreover, $x\in T(F,\theta_F/2)$ implies $\dist(x,F)<\theta_F/2$, so $\dist(z,F)<\e_G + \theta_F/2$. Now there are three sub-cases.

\begin{itemize}
    \item If $F\cap G=\{0\}$, then we readily have a contradiction. Indeed, applying Lemma \ref{Lemma: bounded b.s.}(\ref{Item Lemma: bdd disjoint}) to $z\in\uS_G$ and $F$ gives the absurdity that
    \begin{equation*}
        \frac{1}{kM}\leq \dist(z,F)<\e_G + \theta_F/2 \overset{(\ref{Claim: i)}), (\ref{Claim: ii)})}{\leq} \frac{1}{4kM}+ \frac{1}{8nM} \overset{k \leq n}{\leq} \frac{1}{2kM}.
    \end{equation*}
    \item If $G\subseteq F$, then $\P_F\subseteq \P_G$. However, this contradicts $x\in \P_F\setminus \P_G$ ($x\in S$ with $S\in\Omega_G$ implies $x\notin \P_G$ by using item (\ref{Claim: vi)})).
    \item The last sub-case is that $F\cap G\neq \{0\}$ and $G\not\subseteq F$. In particular, these conditions give $F\cap G \subsetneq G$ (and $F\cap G \subsetneq F$ as well, since $\dim (F\cap G)\leq k-1$). We can now apply Lemma \ref{Lemma: bounded b.s.}(\ref{Item Lemma: bdd general}) to get
    \begin{equation*}\begin{split}
        \dist(z, F\cap G) &\leq kM\cdot \dist(z,F)<kM (\e_G + \theta_F/2) \overset{(\ref{Claim: i)})}{\leq} kM (\e_G + \e_F/2)\\
        &\overset{(\ref{Claim: iii)})}{\leq}kM \left(\frac{1}{4kM}\theta_{F\cap G}+ \frac{1}{8nM}\theta_{F\cap G}\right) \overset{k \leq n}{\leq} \frac{1}{2}\theta_{F\cap G}.
    \end{split}\end{equation*}
    Notice that, when applying condition (\ref{Claim: iii)}) we are using the assumptions that $F\cap G \subsetneq G$ and $F\cap G \subsetneq F$. This estimate implies that $z\in T(F\cap G,\theta_{F\cap G})$; however, this is a contradiction with (\ref{eq: z not in previous tubular nhoods}), since $F\cap G \subsetneq G$.
\end{itemize}

\medskip
\emph{Case 2.} $\dim G=n$ and $\dim F\leq n-1$. 
\smallskip

Here we set $k\coloneqq\dim F\leq n-1$ and, as in Case 1, we assume towards a contradiction that there is $x\in \P_F \cap T(F,\theta_F/2)$ with $x\in S$. As above, $x,z\in S$ and (\ref{Claim: v)}) imply $\|x-z\|<\e_G$; we again distinguish the same three sub-cases.
\begin{itemize}
    \item If $F\cap G=\{0\}$, we use the assumption that $x\in T(F,\theta_F/2)$ to find $w\in F$ with $\|x-w\|<\theta_F/2$. Since $x\in S$, $1-\e_G\leq \|x\| \leq1$, so $1-\e_G-\theta_F/2\leq \|w\| \leq1+ \theta_F/2$. Then the vector $\tilde{w}\coloneqq \frac{w}{\|w\|}$ satisfies $\|w-\tilde{w}\|\leq\e_G +\theta_F/2$, hence $\|x-\tilde{w}\|\leq\e_G +\theta_F$. Moreover $\dist(x,G)\leq\|x-z\|<\e_G$. This finally yields $\dist(\tilde{w},G)\leq 2\e_G +\theta_F$. We can now apply Lemma \ref{Lemma: bounded b.s.}(\ref{Item Lemma: bdd disjoint}) to $\tilde{w}\in\uS_F$ and $G$ to get the following contradiction:
    \begin{equation*}
        \frac{1}{kM}\leq\dist(\tilde{w},G)\leq 2\e_G +\theta_F \leq \frac{2}{4nM}+ \frac{1}{4kM} \leq \frac{3}{4kM},
    \end{equation*}
    where we are using again (\ref{Claim: i)}), (\ref{Claim: ii)}), and that $k\leq n$.
    \item If $F\subseteq G$, then actually $F\subsetneq G$, as $\dim F=k\leq n-1$. The assumption $x\in T(F,\theta_F/2)$ gives $\dist(x,F)<\theta_F/2$, whence (using (\ref{Claim: iii)}) and that $F \subsetneq G$)
    \begin{equation*}
        \dist(z,F)<\e_G + \theta_F/2 \overset{(\ref{Claim: iii)})}{\leq} \frac{1}{4nM}\theta_F + \theta_F/2<\theta_F.
    \end{equation*}
    Hence $z\in T(F,\theta_F)$, which is in contradiction with (\ref{eq: z not in previous tubular nhoods}) since $F\subsetneq G$.
    \item If $F\cap G\neq \{0\}$ and $F\not\subseteq G$, then $F\cap G \subsetneq F$ (and $F\cap G \subsetneq G$ as well since $\dim (F\cap G)\leq k\leq n-1$). As before, we can pick $w\in F$ with $\|x-w\|< \theta_F/2$. Since $\|x-z\|<\e_G$, such a $w$ satisfies $\dist(w,G)\leq \|w-z\|\leq \e_G+\theta_F/2$. Combining the estimate $\dist(w,G)\leq \e_G+\theta_F/2$ with Lemma \ref{Lemma: bounded b.s.}(\ref{Item Lemma: bdd general}) applied to $w\in F$ (the lemma is used in the second inequality) we obtain
    \begin{equation*}\begin{split}
        \dist(z, F\cap G) &\leq \dist(w, F\cap G)+ (\e_G + \theta_F/2)\\
        &\leq kM\cdot \dist(w,G) + (\e_G + \theta_F/2)\\
        &\leq (kM+1)\cdot (\e_G + \theta_F/2) \overset{(\ref{Claim: i)})}{\leq} 2kM \cdot (\e_G + \e_F/2)\\
        &\overset{(\ref{Claim: iii)})}{\leq}2kM \left(\frac{1}{4nM}\theta_{F\cap G}+ \frac{1}{8kM}\theta_{F\cap G}\right) \overset{k\leq n}{\leq} \frac{3}{4}\theta_{F\cap G}.
    \end{split}\end{equation*}
    Again, this implies $z\in T(F\cap G,\theta_{F\cap G})$, a contradiction with (\ref{eq: z not in previous tubular nhoods}).
\end{itemize}
This concludes the induction step, hence the proof of Fact \ref{Fact: compatibility}.
\end{proof}

\begin{step}\label{Step: polyhedral norm} Construction of a polyhedral LFC norm and proof of Theorem \ref{Th A}(\ref{Th A: poly}).
\end{step}

In this step we shall use the sets $(\P_F)_{F\in\F}$ and $(\P_n)_{n\in\N}$ to build a polyhedral and LFC norm on $\Y$ that approximates $\n$. The desired norm will be the Minkowski functional of the set $\P$ defined by
\begin{equation}\label{eq: def P}
    \P\coloneqq\bigcap_{n\in\N} \P_n.
\end{equation}
To begin with, every set $\P_F$ ($F\in\F$) is a closed, convex, and symmetric set (the symmetry is consequence of condition (\ref{Claim: iv)})); plainly, we also have $\P_F\subseteq \B_\Y$. Moreover, by (\ref{Claim: v)}), $\P_F$ is obtained from $\B_\Y$ by removing slices of diameter at most $\e_F\leq \e$, so $(1-\e)\B_\Y\subseteq \P_F$. As a consequence of these remarks, the set $\P$ is a closed, convex and symmetric set with
\begin{equation*}
    (1-\e) \B_\Y \subseteq \P \subseteq \B_\Y.
\end{equation*}
Thus, the Minkowski functional $\mu_\P$ of $\P$ is an equivalent norm on $\Y$, whose unit ball is precisely $\P$. Further, the previous chain of inclusions gives $\n \leq \mu_\P\leq (1-\e)^{-1}\n$, whence $\mu_\P$ approximates $\n$. We now show that $\mu_\P$ is a polyhedral and LFC norm on $\Y$, thereby concluding the proof of Theorem \ref{Th A}(\ref{Th A: poly}). \smallskip

In order to check that $(\Y,\mu_\P)$ is polyhedral, take a finite-dimensional subspace $E$ of $\Y$. Then there is $F\in\F$ with $E\subseteq F$. Hence, it suffices to prove that the unit ball $\P\cap F$ of $F$ is a polyhedron. From Fact \ref{Fact: compatibility} we obtain in particular $\P_n\cap F = \P_F \cap F$ for every $n\in\N$, $n\geq \dim F$; thus, $\P \cap F = \P_F \cap F$. We enumerate all the slices that appear in the definition of $\P_F$:

\begin{equation}\label{eq: enumerate slices}
    \bigcup_{\substack{G\in\F \\ G \subseteq F}}\Omega_G = \{S_j\}_{j=1}^N,
\end{equation}
where $S_j= S(x_j,\psi_j,\delta_j)$. Then we can write
\begin{equation*}\begin{split}
    \P\cap F= \P_F\cap F &=\{y\in\B_F \colon \langle \psi_j, y \rangle\leq 1-\delta_j \mbox{ for all } j=1,\dots, N \}\\
    &=\{y\in F \colon \langle \psi_j, y \rangle\leq 1-\delta_j \mbox{ for all } j=1,\dots, N \}=:\mathcal{C}.
\end{split}\end{equation*}
Indeed, the second equality is just the definition of $\P_F$, while the `$\subseteq$' inclusion of the third equality is obvious. In order to prove that $\mathcal{C}\subseteq \P_F \cap F$, it is sufficient to prove that $\mathcal{C}\subseteq \B_F$. Towards a contradiction, assume that there is $x\in \mathcal{C}$ with $\|x\|>1$. Then $x/\|x\|\in \mathcal{C}$ as well. Notice that the second clause of (\ref{Claim: vii)}) gives in particular $\uS_F\subseteq \cup_{j=1}^N S_j$. Hence, there is $j\in\{1,\dots,N\}$ such that $x/\|x\|\in S_j$, which contradicts $x/\|x\| \in \mathcal{C}$.

Consequently, we obtained that
\begin{equation*}
    \P\cap F = \P_F \cap F = \{y\in F \colon \langle \psi_j, y \rangle\leq 1-\delta_j \mbox{ for all } j=1,\dots, N \}
\end{equation*}
is a polyhedron, hence $\mu_\P$ is a polyhedral norm. \smallskip

Finally, we use Lemma \ref{Lemma: LFC} to show that the norm is LFC. First, in the notation of (\ref{Claim: vii)}), we have $\P_F=\B_\Y \setminus \U_F$, so $\dist(\uS_F, \P_F)\geq 2\theta_F$ for every $F\in\F$. Thus
\begin{equation}\label{eq: PF cap F <1}
    \|x\|\leq 1-2\theta_F \mbox{ for all } x\in \P_F \cap F.
\end{equation}
Moreover, as before, Fact \ref{Fact: compatibility} gives $\P\cap T(F,\theta_F/2) = \P_F \cap T(F,\theta_F/2)$ for every $F\in\F$. 

Now fix arbitrarily $x\in \partial \P$. Then there is $F\in\F$ with $x\in F$; hence $x\in \P_F \cap F$ and $\|x\|\leq 1-2\theta_F$ by (\ref{eq: PF cap F <1}). Consider the open ball $\U\coloneqq \B^o_{\theta_F/2}(x)\coloneqq\{y\in\Y \colon \|y-x\|<\theta_F/2 \}$; then $\U\subseteq T(F,\theta_F/2)$. Indeed, for every $y\in\U$, $\dist(y,F)\leq \|y-x\|<\theta_F/2$ and $\|y\|\leq \|x\| +\theta_F/2 \leq 1-2\theta_F +\theta_F/2 <1$. The inclusion $\U\subseteq T(F,\theta_F/2)$ then allows us to `localise' the compatibility condition and deduce from Fact \ref{Fact: compatibility} that $\P\cap \U = \P_F \cap \U$.

Moreover, using (\ref{eq: enumerate slices}) and the definition of $\P_F$, we can write
\begin{equation*}
    \P_F =\left\{y\in\B_\Y \colon \left\langle \frac{\psi_j}{1-\delta_j}, y \right\rangle \leq1 \mbox{ for all } j=1,\dots, N \right\}.
\end{equation*}

Consequently, for every $y\in\U$, $y\in \P$ if and only if $\left\langle \frac{\psi_j}{1-\delta_j}, y \right\rangle \leq1$ for every $j=1,\dots, N$ (here we are using the equality $\P\cap \U = \P_F \cap \U$). Therefore, Lemma \ref{Lemma: LFC} implies that $\mu_\P$ is LFC, as desired.

\begin{step}\label{Step: smooth norm} Smoothing and proof of Theorem \ref{Th A}(\ref{Th A: smooth}).
\end{step}

In this step we glue together in a smooth way the functionals corresponding to the slices from Step \ref{Step: slices} and obtain a $C^\infty$-smooth and LFC norm on $\Y$. The smoothness of the resulting norm will crucially depend again on the compatibility condition (\ref{eq: compatibility}). In order to leave room for the smoothing, the main technical point consists in actually applying Fact \ref{Fact: compatibility} to some suitably enlarged slices and not to the slices from Step \ref{Step: slices}.\smallskip

Fix $F\in\F$ and write\footnote{Note that, differently from (\ref{eq: enumerate slices}), here we do not enumerate the slices in $\Omega_G$ with $G\subsetneq F$.} $\Omega_F= \left\{\pm S(x_j,\psi_{x_j}, \delta_{x_j}) \right\}_{j=1}^k$. We first define
\begin{equation*}
    \Theta_F \coloneqq\left\{\frac{\psi_{x_j}}{1-\delta_{x_j}}\right\}_{j=1}^k.
\end{equation*}
Observe that this notation allows us to write
\begin{equation}\label{eq: PF via Thetas}
    \P_F\coloneqq\left\{y\in\B_\Y\colon |\langle\psi,y\rangle|\leq1 \mbox{ for all } \psi\in\bigcup_{\substack{G\in\F \\ G \subseteq F}}\Theta_G \right\}.
\end{equation}
Moreover, since all the finitely many slices in $\Omega_F$ have diameter less than $\e_F$, by Fact \ref{Fact: continuity slice} there is $\delta_F>0$ such that $\diam(S(x_j,\psi_{x_j}, \delta_{x_j}+\delta_F))<\e_F$ for every $j=1,\dots,k$. We then pick, for every $F\in\F$, an even, convex, and $C^\infty$-smooth function $\ro_F\colon \R\to[0,\infty)$ such that $\ro_F(1)=1$ and $\ro_F(s)=0$ if and only if $|s|\leq1-\delta_F$. Note that every such a function is strictly increasing on $[1-\delta_F,\infty)$. We are now in position to define $\Phi\colon \Y\to[0,\infty]$ by
\begin{equation*}
    \Phi(y)\coloneqq \sum_{F\in\F}\sum_{\psi\in\Theta_F} \ro_F\big(\langle\psi, y\rangle\big).
\end{equation*}
For $F\in\F$, we also define $\Phi_F\colon \Y \to[0,\infty)$ by
\begin{equation*}
    \Phi_F(y)\coloneqq \sum_{\substack{G\in\F \\ G \subseteq F}} \sum_{\psi\in\Theta_G} \ro_G\big(\langle\psi, y\rangle\big).
\end{equation*}

We wish to apply Lemma \ref{Lemma: implicit mu} to the function $\Phi$ and sets $\D\coloneqq\{\Phi<1\}$ and $\B\coloneqq\{\Phi\leq 1-\e\}$. To this aim we first observe that\footnote{Recall the definition of $\P$ from (\ref{eq: def P}).}
\begin{equation}\label{eq: smooth body approximates}
    (1-\e) \B_\Y \subseteq \{\Phi=0\} \subseteq \B \subseteq \D \subseteq \P \subseteq \B_\Y.
\end{equation}
Indeed, in order to check the first inclusion note that, for $S=S(x_j,\psi_{x_j},\delta_{x_j})\in\Omega_F$, the condition $\diam (S(x_j,\psi_{x_j},\delta_{x_j}+ \delta_F))<\e_F$, together with (\ref{Claim: i)}) of Claim \ref{Claim: slices}, implies $\delta_{x_j}+ \delta_F<\e_F\leq \e$. Hence, if $y\in\Y$ satisfies $\|y\|\leq1-\e$, then for every $F\in\F$ and every $\psi=\frac{\psi_{x_j}}{1-\delta_{x_j}}\in\Theta_F$ we have
\begin{equation*} 
|\langle\psi,y\rangle|= \left|\left\langle \frac{\psi_{x_j}}{1-\delta_{x_j}},y \right\rangle \right|\leq \frac{1-\e}{1-\delta_{x_j}}\leq \frac{1-\delta_F -\delta_{x_j}}{1-\delta_{x_j}} \leq 1-\delta_F.
\end{equation*}

This yields $\Phi(y)=0$ and shows the first inclusion. Next, if $y\in \D$, then for every $\psi\in\Theta_F$ and every $F\in\F$ we have $\ro_F\big(\langle\psi, y\rangle\big) <1$, so $|\langle\psi,y\rangle|<1$ by the properties of the functions $\ro_F$. Hence, $y\in \P_F$ for every $F\in\F$, whence $\D\subseteq \P$. The other inclusions being trivial, (\ref{eq: smooth body approximates}) is proved.

In particular, the set $\B$ is bounded and closed in $\Y$ (since $\Phi$ is lower semi-continuous on $\Y$ by Fatou's lemma). Moreover, $\Phi$ is even and convex, hence $\D$ is convex and symmetric. In order to be able to apply Lemma \ref{Lemma: implicit mu} we need to show that $\D$ is open and that $\Phi$ is $C^\infty$-smooth and LFC on $\D$. All these properties follow rather easily from the next claim, asserting that $\Phi$ is locally a finite sum on $\D$.

\begin{claim}\label{Claim: Phi locally finite sum} For every $y\in \D$ there are an open subset $\U$ of $\Y$ with $y\in \U$ and $F\in\F$ such that $\Phi =\Phi_F$ on $\U$.
\end{claim}

Assuming the validity of the claim, note that the function $\Phi_F$ is clearly $C^\infty$-smooth and LFC on $\Y$ (it is defined via a finite sum). Thus Claim \ref{Claim: Phi locally finite sum} immediately yields that $\Phi$ is $C^\infty$-smooth and LFC on some open neighbourhood of $\D$ in $\Y$. Hence, it also follows that $\D$ is open. We are then in position to apply Lemma \ref{Lemma: implicit mu}, which leads us to the conclusion that $\mu_\B$ is a $C^\infty$-smooth and LFC norm. Finally, (\ref{eq: smooth body approximates}) implies $\n\leq \mu_\B\leq (1-\e)^{-1}\n$, hence $\mu_\B$ approximates the norm $\n$. Consequently, in order to conclude the proof of Theorem \ref{Th A}(\ref{Th A: smooth}), we only need to prove Claim \ref{Claim: Phi locally finite sum}.

\begin{proof}[Proof of Claim \ref{Claim: Phi locally finite sum}] \renewcommand\qedsymbol{$\square$}
Take $F\in\F$ with $y\in F$ and consider the set $\U$ defined by
\begin{equation*}
    \U\coloneqq T(F,\theta_F/2) \cap \left\{z\in\Y\colon \|z\|<1, \, |\langle\psi,z\rangle|<1 \mbox{ for every } \psi\in \bigcup_{\substack{G\in\F \\ G \subseteq F}}\Theta_G \right\}.
\end{equation*}

Observe that $\U$ is indeed open in $\Y$. In fact, $T(F,\theta_F/2)$ is open in $\B_\Y$, so $T(F,\theta_F/2) \cap \{z\in\Y\colon \|z\|<1\}$ is open in $\Y$; this and the fact that $\cup_{\substack{G\in\F \\ G \subseteq F}}\Theta_G$ is a finite set yield that $\U$ is open in $\Y$. Moreover, $y\in\U$. Indeed, the assumption that $y\in \D$ yields $|\langle\psi,y\rangle|<1$ for every $\psi\in\Theta_G$ and every $G\in\F$, $G\subseteq F$. By the same reason (or from (\ref{eq: smooth body approximates})) we also get that $y\in \P_F \cap F$. Hence, $\|y\|<1$, in light of (\ref{eq: PF cap F <1}). Finally, since $y\in \B_F\subseteq T(F,\theta_F/2)$, it follows that $y\in\U$.

We shall now prove that $\Phi=\Phi_F$ on the set $\U$. This amounts to proving that for every $z\in\U$, every $G\in\F$, with $G\not\subseteq F$, and every $\psi\in \Theta_G$ one has $\ro_G(\langle\psi, z\rangle)=0$; equivalently, that $|\langle\psi, z\rangle|\leq 1-\delta_G$.

Now the main idea comes. Fix $G\in\F$ with $G\not\subseteq F$ and write $\Omega_G= \left\{\pm S(x_j,\psi_{x_j}, \delta_{x_j}) \right\}_{j=1}^k$. Set $\tilde{\Omega}_G\coloneqq \left\{\pm S(x_j,\psi_{x_j}, \delta_{x_j}+\delta_G) \right\}_{j=1}^k$ and $\tilde{\Omega}_H=\Omega_H$ for $H\in\F$, $H\neq G$. Then define sets $\tilde{\P}_H$ and $\tilde{\P}_n$ as in (\ref{eq: def P_F}) and (\ref{eq: def P_n}), but replacing the slices $(\Omega_H)_{H\in\F}$ with the slices $(\tilde{\Omega}_H)_{H\in\F}$. Since $G\not\subseteq F$, we have $\tilde{\Omega}_H=\Omega_H$ for every $H\in\F$, $H\subseteq F$; hence $\tilde{\P}_F=\P_F$. Moreover, the system of slices $(\tilde{\Omega}_H)_{H\in\F}$ satisfies all conditions (\ref{Claim: i)})--(\ref{Claim: vii)}) in Claim \ref{Claim: slices} with the same parameters $(\e_H)_{H\in\F}$ and $(\theta_H)_{H\in\F}$. Consequently, the sets $\tilde{\P}_H$ and $\tilde{\P}_n$ also satisfy the compatibility condition (\ref{eq: compatibility}) from Fact \ref{Fact: compatibility}. Hence, given $n\in\N$ with $n\geq \dim F, \dim G$, we have
\begin{equation*}
    \P_F\cap T(F,\theta_F/2)= \tilde{\P}_F\cap T(F,\theta_F/2) \overset{\text{(\ref{eq: compatibility})} }{\subseteq} \tilde{\P}_n \subseteq \tilde{\P}_G.
\end{equation*}
Besides, $\U\subseteq \P_F\cap T(F,\theta_F/2)$ by (\ref{eq: PF via Thetas}). Therefore, for every $z\in\U$ and every slice $\pm S(x_j,\psi_{x_j}, \delta_{x_j}+\delta_G)\in \tilde{\Omega}_G$, we have $z\notin \pm S(x_j,\psi_{x_j}, \delta_{x_j}+\delta_G)$; in other words, $|\langle\psi_{x_j},z \rangle|\leq 1-\delta_{x_j}- \delta_G$, for every $j=1,\dots,k$.

Finally, if $\psi\in\Theta_G$, there is $j\in\{1,\dots, k\}$ with $\psi=\frac{\psi_{x_j}}{1-\delta_{x_j}}$, whence
\begin{equation*}
    |\langle\psi,z\rangle|= \left|\left\langle \frac{\psi_{x_j}}{1-\delta_{x_j}},z \right\rangle \right|\leq \frac{1-\delta_{x_j}- \delta_G}{1- \delta_{x_j}}\leq 1-\delta_G.
\end{equation*}
This yields $\ro_G(\langle\psi,z\rangle)=0$; hence $\Phi(z)=\Phi_F(z)$, as desired.
\end{proof}
As we explained before, this concludes Step \ref{Step: smooth norm} and the proof of Theorem \ref{Th A}(\ref{Th A: smooth}).
\end{proof}

\begin{remark}\label{Rmk: denting LUR} It is apparent from the above proof that we only used the LUR condition via Fact \ref{Fact: small slice}, namely we only used that the fact that if $\Y$ is LUR, then every point of $\uS_\Y$ is strongly exposed. One could wonder whether the argument could be modified as to only require every point of $\uS_\Y$ being denting for $\B_\Y$. However, this is not a more general assumption, since Raja \cite{Raja} proved that every normed space $\Y$ such that every point of $\uS_\Y$ is a denting point for $\B_\Y$ actually has a LUR renorming. Indeed, although \cite[Theorem~1]{Raja} is stated for Banach spaces, it is readily seen that the proof of the implication (iii)$\Rightarrow$(i) there is valid for every normed space. The same result appeared earlier in \cite{T denting LUR}, but the (probabilistic) proof there seems to really depend on completeness.
\end{remark}

As we already mentioned in the Introduction, the following corollary of Theorem \ref{Th A} generalises a result from \cite{Moreno open dense}.

\begin{corollary}\label{Cor: open dense} Let $\X$ be a Banach space with a fundamental biorthogonal system. Then densely many norms on $\X$ are $C^\infty$-smooth and LFC on a dense open subset of $\X$.
\end{corollary}
\begin{proof} Let $\{e_\alpha; \p_\alpha\}_{\alpha \in\Gamma}$ be a fundamental biorthogonal system in $\X$ and let $\Y\coloneqq \spn\{e_\alpha\}_{\alpha\in\Gamma}$. According to Theorem \ref{Th A}(\ref{Th A: smooth}), we can take a $C^\infty$-smooth and LFC norm $\n$ on $\Y$. Then for every $y\in\Y$, $y\neq 0$, there are $\delta_y>0$, functionals $\psi_1,\dots,\psi_n\in \X^*$, and a $C^\infty$-smooth function $G\colon \R^n\to\R$ such that
\begin{equation}\label{eq: smooth locally factorise}
    \|z\|= G\big( \langle\psi_1,z\rangle,\dots, \langle\psi_n,z\rangle \big) \qquad \text{for every } z\in \Bo(y,\delta_y),
\end{equation}
where $\Bo(y,\delta_y):=\{z\in \Y\colon \|y-z\|<\delta_y \}$ (see, \emph{e.g.}, \cite[Fact 5.79]{HJ}). Since obviously $\Bo(y,\delta_y)$ is dense in $\B_\X^{\mathsf{o}}(y,\delta_y)$, the formula (\ref{eq: smooth locally factorise}) is actually valid for every $z\in \B_\X^{\mathsf{o}}(y,\delta_y)$. Consequently, setting $\U:=\cup_{y\in\Y\setminus\{0\}} \B_\X^{\mathsf{o}}(y,\delta_y)$, $\U$ is a dense open set in $\X$, and $\n$ is $C^\infty$-smooth and LFC in $\U$. Finally, $C^\infty$-smooth and LFC norms are dense in $\Y$, so we are done.
\end{proof}

\section{Partitions of unity}\label{Sec: partitions}
In this short section we explain how Theorem \ref{Th A}(\ref{Th A: partitions of unity}) follows from \cite{Haydon} and we begin by recalling Haydon's result from \cite{Haydon} that we need (also see \cite[Theorem 7.53]{HJ}). We denote by $\{e^*_\gamma\}_{\gamma\in\Gamma}$ the coordinate functionals on $c_0(\Gamma)$ and we refer to \cite[\S 7.5]{HJ} for basic definitions concerning partitions of unity.
\begin{theorem}[Haydon, \cite{Haydon}]\label{Th: Haydon} Let $\Y$ be a normed space with a $C^k$-smooth bump function. Assume that:
\begin{romanenumerate}
    \item there is a continuous function $\Phi\colon \Y\to c_0(\Gamma)$ such that $e_\gamma^*\circ\Phi$ is $C^k$-smooth where non-zero, for every $\gamma\in\Gamma$,
    \item for every finite set $F\subseteq \Gamma$ there is a $C^k$-smooth map $P_F\colon \Y\to\Y$ such that $\spn (P_F(\Y))$ has locally finite $C^k$-smooth partitions of unity,
    \item for every $x\in\Y$ and $\e>0$ there is $\delta>0$ such that $\|x-P_F (x)\|<\e$, where $F\coloneqq\{\gamma\in\Gamma\colon |\Phi(x)(\gamma)|\geq\delta\}$.
\end{romanenumerate}
Then $\Y$ admits locally finite and $\sigma$-uniformly discrete $C^k$-smooth partitions of unity.
\end{theorem}

In the case when $\Y\coloneqq\spn\{e_\alpha\} _{\alpha\in\Gamma}$ is the linear span of some fundamental biorthogonal system $\{e_\alpha;\p_\alpha\}_{\alpha\in \Gamma}$ we can apply Haydon's result as follows. The map $\Phi$ is defined by
\begin{equation*}
    \Phi(x)\coloneqq\left(\left\langle\frac{\p_\alpha}{\|\p_\alpha\|},x\right\rangle\right)_{\alpha\in\Gamma},
\end{equation*}
so that $e^*_\gamma\circ\Phi$ is evidently $C^\infty$-smooth and LFC on $\Y$. The map $P_F$ is just the canonical projection from $\Y$ onto $\spn\{e_\alpha\}_{\alpha\in F}$; surely, the finite-dimensional normed space $\spn\{e_\alpha\}_{\alpha\in F}$ admits locally finite $C^\infty$-smooth and LFC partitions of unity. The approximation condition is also easily satisfied, since for every $x\in\Y$ there is a finite set $F\subseteq\Gamma$ with $x=P_F(x)$. Hence Theorem \ref{Th: Haydon} yields that $\Y$ admits locally finite and $\sigma$-uniformly discrete $C^\infty$-smooth partitions of unity. 

In order to explain why such partitions of unity are also LFC, we follow the proof and notation in \cite[Theorem 7.53]{HJ}. Since the functions $\p_{F,q,r}$ defined there are LFC and $\Phi$ is linear, every map $\p_{F,q,r}\circ\Phi$ is LFC (see, for example, \cite[Fact 5.80]{HJ}). Next, instead of using the partition ring $C^\infty(\mathcal{Z})$, where $\mathcal{Z}$ is a normed space, we replace it with the partition ring of $C^\infty$-smooth LFC functions on $\mathcal{Z}$ (we refer to \cite[Definition 7.47]{HJ} for the definition of partition ring). Since the partition ring of $C^\infty$-smooth LFC functions on $\mathcal{Z}$ is determined locally, in the sense of \cite[Definition 7.48]{HJ}, all the statements in \cite[Lemma 7.49]{HJ} are equivalent. Also notice that the partition ring of $C^\infty$-smooth LFC functions on $\Y$ contains a bump function by Theorem \ref{Th A}(\ref{Th A: bump}). After these remarks, we can return to the argument in \cite[Theorem 7.53]{HJ}. The unique additional difference is that we need to show that each of the sets $\Phi^{-1}(W_{F,q,r})$, $P_F^{-1}(V)$, and $(Id-P_F)^{-1}(U_m)$ is of the form $\{f\neq 0\}$, for some $C^\infty$-smooth and LFC function $f\colon\Y\to \R$. We explain this for sets $P_F^{-1}(V)$, the other two cases being analogous. By assumption $V=\{g\neq0\}$, for some $C^\infty$-smooth and LFC function $g\colon \spn\{e_\alpha\}_{\alpha\in F}\to \R$. Then $P_F^{-1}(V)=\{g\circ P_F\neq 0\}$ and the function $g\circ P_F\colon \Y\to\R$ is $C^\infty$-smooth and LFC (again, the LFC property follows from the linearity of $P_F$ \footnote{This argument does not work in the general setting of Theorem \ref{Th: Haydon}, since there the maps $\Phi$ and $P_F$ are not necessarily linear.}). All the remaining steps of the proof being identical to the argument in \cite{HJ}, we conclude the validity of Theorem \ref{Th A}(\ref{Th A: partitions of unity}).

\section{\texorpdfstring{$C^1$}{C¹}-smooth LUR norms}\label{Sec: C1 LUR}
In this section we discuss the proof of Theorem \ref{Th A}(\ref{Th A: C1 LUR}), thereby concluding the proof of our main result. As we said already, the argument is an adaptation of \cite{HP}, therefore we shall restrict ourselves to defining the desired $C^1$-smooth LUR norm and refer to \cite{HP} for the verification of the various properties. Very roughly speaking, the main idea in \cite{HP} is to smoothen up the formula for the norm from Troyanski's renorming technique. Since LUR renormings typically involve countably many contributions around each point, it is crucial to have uniform Lipschitz estimates in order to obtain $C^1$-smoothness in the limit. Two earlier results, based on the same strategy but substantially less technical, are \cite{MPVZ} and \cite{PWZ} (also see \cite[\S V.1]{DGZ}), where a $C^1$-smooth LUR norm is constructed in every separable Banach space and in $c_0(\Gamma)$, respectively. \smallskip

If $\Y\coloneqq \spn\{e_\alpha\} _{\alpha\in\Gamma}$ is the linear span of a fundamental biorthogonal system $\{e_\alpha;\p_\alpha \}_{\alpha \in\Gamma}$, then, to begin with, we can assume that $\{e_\alpha;\p_\alpha \}_{\alpha \in\Gamma}$ is bounded, thanks to Theorem \ref{Th: fundamental b.s.}(\ref{Item Th: bdd system}). Secondly, in light of Theorem \ref{Th A}(\ref{Th A: smooth}), $C^\infty$-smooth norms are dense in the set of all equivalent norms on $\Y$. Hence Theorem \ref{Th A}(\ref{Th A: C1 LUR}) is a consequence of Theorem \ref{Th: C1 LUR approx} below.

\begin{theorem} \label{Th: C1 LUR approx} Let $\X$ be a Banach space with a bounded fundamental biorthogonal system $\{e_\alpha;\p_\alpha \}_{\alpha \in\Gamma}$. Let $\Y\coloneqq \spn\{e_\alpha\} _{\alpha\in\Gamma}$ and let $\n$ be a $C^1$-smooth norm on $\Y$. Then $\n$ can be approximated by $C^1$-smooth LUR norms.
\end{theorem}

Before the proof we need to recall a couple of definitions. Let $\Y$ be a normed space and $\D$ be a convex subset of $\Y$. A convex function $f\colon \D\to \R$ is \emph{strictly convex} if its graph contains no non-trivial segments. Given a set $\Lambda$, a function $g=(g_\gamma)_{\gamma \in\Lambda} \colon \Y\to \ell_\infty (\Lambda)$ is \emph{coordinate-wise convex} (resp.~\emph{coordinate-wise $C^k$-smooth}) if, for every $\gamma\in\Lambda$, the function $g_\gamma$ is convex (resp.~$C^k$-smooth).

\begin{proof}[Proof of Theorem \ref{Th: C1 LUR approx}] Without loss of generality, we can assume that $\|e_\alpha\|=1$ and $\|\p_\alpha\|\leq M$, for some $M\geq 1$ ($M$ corresponds to the constant $2C$ in \cite[\S 4.2]{HP}). We begin by fixing some notation. Let, for every $n\in\N$, $\xi_n\colon [0,\infty)\to [0,\infty)$ be a $C^\infty$-smooth, $1$-Lipschitz convex function such that $\xi_n(t)=0$ when $t\in[0,1/n]$ and $\xi_n(t)=t- \frac{2}{n}$ when $t\geq 3/n$ (take the convolution of $t\mapsto \max\{0,t-\frac{2}{n}\}$ with a suitable bump). Let $\Bo$ be the open unit ball of $\Y$. Given a set $\Lambda$, on $\ell_\infty (\Lambda)$ we consider the seminorm $\ceil\cdot$ defined by
\begin{equation*}
    \ceil z \coloneqq \inf \big\{t>0\colon \{\gamma\in\Lambda \colon |z(\gamma)|>t\} \textrm{ is finite}\big\} = \|q(z)\|_{\ell_\infty(\Lambda) / c_0(\Lambda)},
\end{equation*}
where $q\colon \ell_\infty(\Lambda) \to \ell_\infty(\Lambda) / c_0(\Lambda)$ is the quotient map. For $\eta\in(0,1)$ we consider the set
\begin{equation*}
    A_\eta(\Lambda)\coloneqq \left\{z\in\ell_\infty (\Lambda)\colon \ceil z < (1-\eta) \|z\|_\infty \right\}.
\end{equation*}
Note that $z\in A_\eta(\Lambda)$ is a `strong maximum' condition, in the sense that finitely many coordinates of $z$ are quantitatively larger than the others. The construction of the norm is then performed in three steps. \smallskip

First, consider the system of functions $\left\{g_{n,m,l} \colon n,m,l\in\N, l\leq n \right\}$ on $\R^2$ obtained by shifting and scaling a certain function $g$. Set $g\colon \R^2\to [0,\infty)$ to be $g(t,s)=0$ for $t\leq 0$ and, for $t>0$,
\begin{equation*}
    g(t,s)\coloneqq \exp{(-10/t)} \cdot \left( \frac{s^2}{100} + \frac{s}{10}+1 \right).
\end{equation*}
For $n,m,l\in\N$ with $l\leq n$ define
\begin{equation*}
    g_{n,m,l}(t,s)\coloneqq g\left(\frac{t-l/n}{1+nM}, \theta_{n,m}\frac{s}{1+nM} \right).
\end{equation*}
The main properties of the system $\left\{g_{n,m,l} \right\}$ are listed in \cite[Lemma 4.5]{HP}, where in particular the parameters $\rho_n\in (0,1/2)$, $\theta_{n,m}\in(0,1)$, and $\kappa_{n,m}\in(0,\rho_n)$ are fixed. Let $\eta_{n,m}\coloneqq \rho_n- \kappa_{n,m}\in(0,1/2)$. Consider the set
\begin{equation*}
    \Lambda_n\coloneqq \big\{(A,B)\colon \emptyset\neq B\subseteq A\subseteq \Gamma, |A|\leq n\big\}
\end{equation*}
and the functions $H_{n,m}\colon \Bo\to \ell_\infty(\Lambda_n)$ defined by
\begin{equation*}
    H_{n,m}y(A,B)\coloneqq g_{n,m,|A|}\left(\sum_{\alpha \in A}\xi_n \big(|\langle\p_\alpha, y\rangle|\big), \xi_n \left( \left\|y-\sum_{\alpha\in B} \langle\p_\alpha, y\rangle e_\alpha \right\| \right) \right).
\end{equation*}
It is easily seen that $H_{n,m}$ is $1$-Lipschitz, coordinate-wise convex and coordinate-wise $C^1$-smooth. Additionally, \cite[Lemma 4.7]{HP} asserts that $H_{n,m}y\in A_{\eta_{n,m}}(\Lambda_n)$ for every $y\in \Bo$ with $H_{n,m}y\neq 0$. \smallskip

The second step consists in building the norms for gluing together the functions $H_{n,m}$ in the standard way. For $\eta\in(0,1/2)$, take a $C^\infty$-smooth convex function $\psi_\eta\colon [0,\infty)\to[0,\infty)$ such that $\psi_\eta(t)=0$ for $t\in[0,1-\eta]$, $\psi_\eta$ is strictly convex on $[1-\eta,\infty)$, and $\psi_\eta(1)=1$. We also require that $\psi_{\eta_1}(t) \leq \psi_{\eta_2}(t)$ for $t\in[0,1]$ and $\eta_1< \eta_2$. Set $\Phi_\eta\colon \ell_\infty(\Lambda)\to [0,\infty]$ by
\begin{equation*}
    \Phi_\eta(z)\coloneqq \sum_{\gamma\in \Lambda} \psi_\eta(|z(\gamma)|)
\end{equation*}
and let $Z_\eta$ be the Minkowski functional of the set $\{\Phi_\eta\leq1\}$. One readily sees that $Z_{\eta_1}\leq Z_{\eta_2}$ if $\eta_1\leq\eta_2$, that $Z_\eta$ is a lattice norm, and that $(1-\eta)Z_\eta\leq \n_\infty\leq Z_\eta$. In particular, $Z_\eta$ is $2$-Lipschitz on $\ell_\infty(\Lambda)$. Moreover, $Z_\eta$ is $C^\infty$-smooth and LFC on the set $A_\eta(\Lambda)$ and $(1-\eta)Z_\eta(z) < \|z\|_\infty$ for every $z\in A_\eta(\Lambda)$, \cite[Lemma 4.1]{HP}. Additionally, $Z_\eta$ satisfies a LUR condition for `large coordinates' (namely for those coordinates $\gamma$ for which $z(\gamma)>(1-\gamma) Z_\eta(z)$), \cite[Lemma 4.3 and Lemma 4.4]{HP}. \smallskip

Finally, in the last step we glue all the ingredients together. Consider the norm $Z_\eta$ on $\ell_\infty(\Lambda)$ with $\eta=\eta_{n,m}$ and $\Lambda= \Lambda_n$ as defined above. For $j,n,m\in\N$ define $J_{j,n,m}\colon \Bo\to [0,\infty)$ by
\begin{equation*}
    J_{j,n,m}\coloneqq \xi_j \circ  Z_{\eta_{n,m}} \circ H_{n,m}.
\end{equation*}
Clearly, $J_{j,n,m}$ is $2$-Lipschitz on $\Bo$ and $J_{j,n,m}(0)=0$. Moreover, $Z_{\eta_{n,m}} \circ H_{n,m}$ is $C^1$-smooth on the set $\{H_{n,m} \neq0\}$; this follows from the facts that $H_{n,m}$ is coordinate-wise $C^1$-smooth, that $H_{n,m}y\in A_{\eta_{n,m}}(\Lambda_n)$ when $H_{n,m}y\neq 0$, and that $Z_{\eta_{n,m}}$ is $C^1$-smooth on $A_{\eta_{n,m}}(\Lambda_n)$. Hence, $J_{j,n,m}$ is $C^1$-smooth on $\Bo$. Finally, fix $\e>0$ and define $J\colon \Bo\to[0,\infty)$ by
\begin{equation*}
    J(y)^2\coloneqq \|y\|^2+ \e \sum_{j,n,m\in\N} 2^{-(j+n+m)} J_{j,n,m}(y)^2.
\end{equation*}
Since each $J_{j,n,m}$ is $2$-Lipschitz, the series of the derivatives converges and $J$ is $C^1$-smooth on $\Bo$. Also, $\|y\|\leq J(y)\leq \sqrt{1+4\e}\|y\|$, whence
\begin{equation*}
    \frac{1-\e}{\sqrt{1+4\e}}\B_\Y \subseteq \{J\leq1-\e\}\subseteq \B_\Y.
\end{equation*}
Therefore, the Minkowski functional $\nn\cdot$ of the set $\{J\leq 1-\e\}$ is a norm that approximates $\n$. Moreover, $J$ is $C^1$-smooth on $\Bo\supseteq \{J<1\}$, so Lemma \ref{Lemma: implicit mu} yields that $\nn\cdot$ is $C^1$-smooth. What remains to be proved is that $\nn\cdot$ is LUR, which follows the argument in \cite[Proposition 4.10]{HP}. As a matter of fact the argument is even simpler in our context, since instead of finding a finite set $A$ with 
\begin{equation*}
    \left\|y - \sum_{\gamma\in A} Q_\gamma y \right\|<\e,
\end{equation*}
we can find a finite set $A=\supp (y)$ with $y = \sum_{\alpha\in A} \langle\p_\alpha,y\rangle e_\alpha$ ($Q_\gamma$ in \cite{HP} are projections that correspond to the rank-one projections $\langle\p_\alpha,\cdot\rangle e_\alpha$).
\end{proof}

\subsection{Higher order smoothness and super-reflexivity}
We shall conclude the section observing that the $C^1$-smooth LUR norms that we constructed above can't in general admit any higher order smoothness. This essentially follows from results in \cite{FWZ} (also see \cite[\S 5.2]{HJ}) and we sketch the explanation below. \smallskip

We write that a norm on $\X$ is $C^{1,+}_{\textrm{loc}}$-smooth (resp.~$C^{1,+}$-smooth) if it is differentiable with locally uniformly continuous (resp.~uniformly continuous) derivative on $\uS_\X$. Moreover, we also recall that the \emph{modulus of smoothness} of a norm $\n$ on a Banach space $\X$ is the function $\rho_\X\colon (0,\infty)\to (0,\infty)$ defined by
\begin{equation*}
    \rho_\X(\tau)\coloneqq\sup\left\{\frac{\|x+\tau h\| + \|x-\tau h\|-2}{2}\colon \|x\|=\|h\|=1 \right\}.
\end{equation*}
A Banach space $\X$ is \emph{uniformly smooth} if $\frac{\rho_\X(\tau)}{\tau}\to0$ as $\tau\to 0^+$; it is a classical result that uniformly smooth Banach spaces are super-reflexive.

\begin{theorem}\label{Thm: C2 LUR super-reflexive} Let $\Y$ be a normed space with a $C^{1,+}_{\normalfont\textrm{loc}}$-smooth LUR norm $\n$. Then the completion $\hat{\Y}$ of $\Y$ is super-reflexive.
\end{theorem}

\begin{proof} Since $\n$ is LUR, every point of its unit sphere is strongly exposed. Hence \cite[Theorem 5.46]{HJ} (which comes from \cite{FWZ}) implies that $\Y$ has a $C^{1,+}$-smooth norm $\nn\cdot$. Let $g\colon \Y\setminus\{0\}\to\R$ be the derivative of $\nn\cdot$; by homogeneity, $g$ is uniformly continuous on $\{y\in\Y\colon \nn{y}\geq1/2\}$ with modulus of continuity, say $\omega_g$. Fix $y,h\in\Y$ with $\nn{y}= \nn{h}=1$ and $\tau\in (0,1/2)$. By Lagrange's theorem, there are $\theta^\pm\in(0,1)$ with
\begin{equation*}
    \nn{y\pm \tau h}- \nn{y}= \langle g(y\pm \theta^\pm \tau h), \pm\tau h\rangle   
\end{equation*}
(here we are using the facts that $\nn\cdot$ is differentiable on $\Y$ and $y,h\in\Y$). Note that $\nn{y\pm \theta^\pm \tau h}\geq1/2$. Then we have
\begin{equation*}\begin{split}
    \nn{y+\tau h}+ \nn{y-\tau h} -2\nn{y}&= \left\langle g(y+ \theta^+ \tau h)- g(y- \theta^- \tau h),\tau h \right\rangle\\
    &\leq \nn{g(y+ \theta^+ \tau h)- g(y- \theta^- \tau h)}\cdot\tau \\
    &\leq \omega_g\left(\nn{\theta^+ \tau h+ \theta^- \tau h} \right)\cdot \tau\leq \omega_g(2\tau)\cdot \tau.
\end{split}\end{equation*}
Since $\Y$ is dense in $\hat{\Y}$, the previous inequality is also valid for every $y,h\in\hat{\Y}$ with $\nn{y}= \nn{h}=1$. Hence, (dividing by $2\tau$ and) passing to the supremum over such $y,h$ yields
\begin{equation*}
    \frac{\rho_{\hat{\Y}}(\tau)}{\tau}\leq \frac{\omega_g(2\tau)}{2}\to 0, \qquad\text{as } \tau\to0^+.
\end{equation*}
Consequently, $(\hat{\Y},\nn\cdot)$ is uniformly smooth, as desired\footnote{This argument showing that $C^{1,+}$-smooth norm are uniformly smooth certainly is a classical one. We presented it just because of the passage to the completion that we needed in the midst of it.}.
\end{proof}

\begin{remark} In particular, no dense subspace $\Y$ of $c_0(\Gamma)$ admits a $C^{1,+}_{\textrm{loc}}$-smooth LUR norm. This particular case could also be proved by combining \cite[Theorem 5.46]{HJ} with \cite{W}. More precisely, a standard `small perturbations' argument shows that it is sufficient to consider $\Gamma=\N$ and $\Y=c_{00}$ (see, \emph{e.g.}, \cite[Theorem 2.1]{HR densely}). Then \cite[Theorem 5.46]{HJ} would imply that $c_{00}$ has a $C^{1,+}$-smooth norm, which is false, \cite{W}.
\end{remark}

\medskip{}
\textbf{Acknowledgements.} We are indebted to Simone Ferrari for his helpful comments and references concerning Remark \ref{Rmk: denting LUR} and to Vicente Montesinos for his remarks on \cite{HM}. Finally, we wish to thank the anonymous referees for their detailed and helpful report.


\end{document}